\documentclass[12pt]{article}
\usepackage[T1]{fontenc}
\usepackage[utf8]{inputenc}
\usepackage[numbers]{natbib}
\usepackage[english]{babel}
\usepackage{amsmath,enumerate}
\usepackage{amssymb,bm}
\usepackage[breakable,skins]{tcolorbox}
\usepackage{authblk}
\usepackage{hyperref}
\hypersetup{
    colorlinks=true,
    linkcolor=blue,  
    citecolor=blue,  
    urlcolor=black  
}
\usepackage{tikz,multicol}
\usepackage[left=2.5cm, right=2.5cm, top=2.5cm, bottom=2.5cm]{geometry}

\usepackage{mathrsfs,xcolor}  
\everymath{\displaystyle}
\usepackage{indentfirst}

\hypersetup{colorlinks=true,linkcolor=blue}

\usepackage{amsthm}
\theoremstyle{definition}
\newtheorem{Teorema}{Theorem}[section]
\newtheorem{Corollary}[Teorema]{Corollary} 
\newtheorem{Lemma}[Teorema]{Lemma} 
\newtheorem{Ejemplo}[Teorema]{Example} 
 
\newtheorem{Definicion}[Teorema]{Definition} 
\newtheorem{Proposition}[Teorema]{Proposition} 
\newtheorem{theoremx}{Theorem}
 
\newtheorem{Remark}[Teorema]{Remark}

\newtheoremstyle{italichead}{3pt}{3pt}{\normalfont}{}{\itshape}{.}{ }{}     

\theoremstyle{italichead}

\usepackage{setspace}

\newcounter{step}
\newcounter{case}
\newcounter{subcase}[case]
\newcounter{subsubcase}[subcase]

\renewcommand{\thestep}{\arabic{step}}
\renewcommand{\thecase}{\arabic{case}}
\renewcommand{\thesubcase}{\thecase.\arabic{subcase}}
\renewcommand{\thesubsubcase}{\thesubcase.\arabic{subsubcase}}

\newcommand{\step}[1][]{
  \refstepcounter{step}
  \textbf{Step \thestep#1.}
}

\newcommand{\case}[1][]{
  \refstepcounter{case}
  \textbf{Case \thecase#1.}
}

\newcommand{\subcase}[1][]{
  \refstepcounter{subcase}
  \textbf{Case \thesubcase#1.}
}

\newcommand{\subsubcase}[1][]{
  \refstepcounter{subsubcase}
  \textbf{Case \thesubsubcase#1.}
}

\newcommand{\Canf}[2]{\operatorname{Can}(#1,#2,\prec)}

\newcommand{\Co}{\mathcal{C}}

\newcommand{\Fn}[1]{\mathbb{F}_{#1}}
\newcommand{\F}[2]{\mathbb{F}_{#1}^{#2}}
\newcommand{\D}[1]{\Delta{\bm{#1}}}
\newcommand{\xp}[1]{\bm{x}^{\D{#1}}}
\newcommand{\xpp}[2]{\xp{#1}\cdot \xp{#2}}

\newcommand{\Rx}{\mathcal{R}_{\bm{X}}}
\newcommand{\w}[1]{\omega(\bm{#1})}
\newcommand{\wa}[1]{\omega(#1)}

\newcommand{\wev}[1]{\omega(\langle #1 \rangle)}
\newcommand{\ev}[1]{\langle #1 \rangle }
\newcommand{\dev}[1]{\dim(\langle \bm{#1} \rangle)}
\newcommand{\supp}[1]{\operatorname{supp}(\bm{#1})}
\newcommand{\suppa}[1]{\operatorname{supp}(#1)}
\newcommand{\crc}[1]{\mathcal{M}_{#1}}
\newcommand{\inic}{\operatorname{in}}
\newcommand{\pd}{\operatorname{pd}}
\newcommand{\lcm}{\operatorname{lcm}}
\newcommand{\inicC}[1][\Co]{\inic(I(#1))}
\newcommand{\Kx}{K[\bm{x}]}
\newcommand{\G}{\mathcal{G}}

\newcommand{\ideal}[1]{\mathcal{I}_{#1}}
\newcommand{\GRx}{\G\setminus\Rx}
\newcommand{\Ic}{I(\Co)}
\newcommand{\mun}{\bm{m}_1}
\newcommand{\md}{\bm{m}_2}

\newcommand{\pe}[1]{\left\lfloor \frac{#1}{2} \right\rfloor }
\newcommand{\pep}[1]{\left\lfloor \tfrac{#1}{2} \right\rfloor }

\newcommand{\pecp}[2]{\left\lfloor \tfrac{#1}{#2} \right\rfloor }
\newcommand{\ii}{2}
\newcommand{\I}[1][\Co]{I_{#1}}
\newcommand{\J}[1][\Co]{J_{#1}}

\allowdisplaybreaks

\begin{document}

\title{Gr{\"o}bner bases and the second generalized Hamming weight of a linear code}
\author[1]{Hern\'an de Alba \thanks{Partially supported by SECIHTI Grant CF-2023-G-33}}
\author[2]{Cecilia Mart\'inez-Reyes \thanks{Supported by SECIHTI Grants 4003512 and CF-2023-G-33}}

\affil[1]{SECIHTI - Universidad Aut\'onoma de Zacatecas\\ {\small \texttt{halba@uaz.edu.mx}}}
\affil[2]{Universidad Aut\'onoma de Zacatecas, M\'exico\\{\small\texttt{maria.reyes1@ues.edu.sv}}}
\date{}

\begingroup
\renewcommand\thefootnote{} 
\endgroup

\maketitle
\begin{abstract}
It is known that for binary codes one can use Gr\"obner bases to obtain a subset of codewords of minimal support that can be used to determine the second generalized Hamming weight of the code. In this paper we establish conditions on a nonbinary code under which the same property holds. We also  construct a family of codes over any nonbinary finite field where the  property does not hold. Furthermore, we prove that whenever the subset obtained via Gr\"obner basis suffices to determine the second generalized Hamming weight, this invariant can also be recovered from the degrees of the syzygies of a minimal free resolution. 
\end{abstract}

\section{Introduction}
The Generalized Hamming Weights of a linear code (GHWs) were introduced in \cite{wei1991}, and since then there has been a growing interest in their study due to their applications within information theory. The GHWs have been investigated for special classes of codes such as cyclic codes \citep{yang2015generalized,feng2002generalized,janwa2002generalized,xiong2016weight,heng2017complete,luo2008weight}, Reed-Muller codes \citep{gonzalez2020generalized,heijnen2002generalized,ramkumar2018determining,geil2008second,martinez2020generalized}, trace codes \citep{stichtenoth2002generalized,guneri2006improvements}, among others. The approaches through which they have been studied are diverse, for example, by employing quadratic forms \citep{li2020weight,liu2023generalized}, Gr{\"o}bner bases \citep{borges2008grobner,marquez-corbella_ideal_2016,garcia2022free} or free resolutions \citep{garcia2022free,johnsen2013hamming}. Of particular interest has been the second GHW \citep{carvalho2013second,chung19912,olaya2015second,sarabia2018second,shim1995second,garcia2022free}. 

Given a linear code $\Co\subset\F{q}{n}$, let $\crc{\Co}$ denote the set of codewords of minimal support of $\Co$. In \citet{johnsen2013hamming}, it is proved that all the GHWs of $\Co$ can be recovered using the minimal free resolution of the monomial ideal associated to $\crc{\Co}$. More recently, in \citet{garcia2022free} it is proved that for linear codes $\Co\subset\F{2}{n}$, one can recover the first and second generalized Hamming weight of $\Co$ using a subset of $\crc{\Co}$. 

The technique used in \cite{garcia2022free} consists in considering a specific subset $M\subset \crc{\Co}$ where one can find two codewords that form a two-dimensional vector space, whose weight is the second GHW of the linear code. In this work, we refer to any subset of  $\crc{\Co}$ that satisfies this property as a $d_2$-test set of $\Co$.  With this terminology, it was proved that for binary codes $\Co$, the set of codewords associated to the reduced Gr{\"o}bner basis of a binomial ideal that can be associated to $\Co$, is a $d_2$-test set. The authors also formulated several questions pointing towards directions of research on the computation of GHWs using Gr\"obner bases.  

Of particular interest for the present work is Question 4 from  \cite{garcia2022free}, which asks whether further results can be obtained in the nonbinary case. To address this question, it was suggested to consider a binomial ideal $\Ic$ from \cite{marquez-corbella_ideal_2016}, which  generalizes the binomial ideal of the binary case.  Hence, given $\Co\subset\F{q}{n}$, we investigate conditions under which the set $M_{\G}$, consisting of the codewords of minimal support associated to the reduced Gr{\"o}bner basis of $\Ic$, is a $d_2$-test set.

More precisely, for any linear code $\Co\subset\F{q}{n}$, we consider two codewords $\mun,\md\in\Co$  such that the weight of the subspace spanned by $\mun$ and $\md$ is the second GHW of $\Co$, under additional conditions to be specified later. Also, we consider the reduced Gr\"obner basis $\G$ of $\Ic$ with respect to degree compatible monomial orders with appropriate structural properties. Within this framework, our first main theorem is (Theorem~\ref{corollary-main}):

\begin{theoremx}\label{in:g-d2-test}
Let $\Co\subset\F{q}{n}$ be a linear code such that $|\suppa{\mun} \cap \suppa{\md}|\leq \frac{\wa{\md}+1}{2}$. Let $\G$ be the reduced Gr\"obner basis of $\Ic$. Then, $M_{\G}$ is a $d_2$-test set, i.e., there exist $\bm{c}_1,\bm{c}_2\in M_\G$ such that $\dim\ev{\bm{c}_1,\bm{c}_2}=2$ and $d_2(\Co)=\wev{\bm{c}_1,\bm{c}_2}$. 
\end{theoremx}

As we mentioned before, in the binary case $M_\G$ is a $d_2$-test set. Unfortunately, this is no longer true in general. In this work, for each $q>2$ we provide a family of linear codes $\Co\subset\F{q}{n}$ where $M_\G$ is not a $d_2$-test set. This establishes our second main theorem (Theorem~\ref{main-result}):

\begin{theoremx}\label{in:g-d2-no-test}
For each $q>2$, there exist linear codes $\Co\subset\F{q}{n}$ such that $M_{\G}$ is not a $d_2$-test set.
\end{theoremx}

Finally, we show that whenever $M_\G$ is a $d_2$-test set of a linear code $\Co\subset \F{q}{n}$, the minimum Hamming weight $d_1(\Co)$ and the second GHW $d_2(\Co)$ can be determined using free resolutions in a way similar to the approach of \cite{johnsen2013hamming} and \cite{garcia2022free}. More precisely, our third main result establishes the following (Theorem~\ref{ym-free-re}):

\begin{theoremx}\label{in:g-free}
Let $\Co \subset \F{q}{n}$ be a linear code of dimension $k$. 
Let $M \subset \crc{\Co}$. Let $K$ be a field and let $R = K[x_1, x_2, \ldots, x_n]$.
Define $S_M=\{\sigma:\exists \bm{c}\in M\text{ s. t. }\supp{c}=\sigma\}$
where $|S_M| \geq 2$. 
Let $\ideal{M}=\langle \bm{x}^{\sigma}:\sigma\in S_M \rangle \subset R$. Then

\begin{enumerate}
\item $\min \{j \mid \beta_{1,j}(R/\ideal{M})\neq 0\}=d_1(\Co)$  if and only if $M$ contains a codeword of minimum Hamming weight.
\item $\min \{j \mid \beta_{2,j}(R/\ideal{M})\neq 0\}=d_2(\Co)$ if and only if $M$ is a $d_2$-test set.
\end{enumerate}
Where $\beta_{i,j}(R/\ideal{M})$ are the Betti numbers of the graded minimal free resolution of $R/\ideal{M}$. 
\end{theoremx}

\medskip

Thus, by combining Theorems \ref{in:g-d2-no-test} and \ref{in:g-free}, we conclude that, in general, the minimal free resolution of $R/I_{M_\G}$ does not always allow $d_2(\Co)$ to be recovered in the same way as in the binary case. 

\medskip

The paper is organized as follows. Section~\ref{s:premilinaries} collects the notation and preliminary material required throughout the article. In Section ~\ref{s:d2-test-set} we introduce the notion of a $d_2$-test set and define the sets that allow us to formulate a sufficient condition for a subset of codewords with minimal support to be a $d_2$-test set. Section~\ref{s:mg-si-d2} specializes this result by providing a sufficient condition for $M_\G$  to be a $d_2$-test set. In Section~\ref{s:mg-no-d2}, for any $q> 2$, we construct a family of linear codes $\Co\subset\F{q}{n}$ where $M_\G$ fails to be a  $d_2$-test set. Finally, Section~\ref{s:free-res} investigates the minimal free resolution of a monomial ideal associated to a subset $M$ of $\crc{\Co}$, and shows how the degrees of their syzygies are related to the condition that $M$ is a $d_2$-test set.

\section{Preliminaries}\label{s:premilinaries}

\subsection{Basic notions}

Let $\Co\subset\F{q}{n}$ be a linear code of dimension $k$. The elements in $\F{q}{n}$ are called words, and the elements in $\Co$ are called codewords of $\Co$.  For clarity, we may write $\Co(n,k)_q$ to explicitly indicate the dimension of $\Co$ and the vector space in which it lies, but, in general, we simply write $\Co$ when no confusion arises. A matrix $G \in \Fn{q}^{k \times n}$ is defined as a generator matrix of $\Co$ if its rows form a basis for $\Co$, while a matrix $H \in \Fn{q}^{(n-k) \times n}$ is called a parity check matrix of $\Co$ if the null space of $H$ coincides with $\Co$.
If $V\subset\F{q}{n}$, we write $\ev{V}$ to denote the vector space generated by $V$ over the field $\Fn{q}$. For $r\in \mathbb{N}$, we use $[r]$ to denote the set $\{1,2,\ldots,r\}$.

Given $\bm{a}=(a_1,a_2,\ldots,a_n)\in\F{q}{n}$, its support is defined as $\supp{a}=\{i\in [n]:a_i\neq 0\}$. The Hamming weight of $\bm{a}$ is the cardinality of $\supp{a}$ and is denoted by $\w{a}$. The minimum Hamming weight of a linear code $\Co$ is $d(\Co)=\min \{\w{a}:\bm{a}\in\Co\setminus\{\bm{0}\}\}$.

If $\Co(n,k)_q$ is a linear code, and $D$ is a vector subspace of $\Co$, the support of $D$ is the set $\suppa{D}=\{i\in[n]:\exists \bm{c}\in D \text{ such that }i\in \supp{c}\}$, and the weight of $D$ is $\wa{D}=|\suppa{D}|$. For $i\in[k]$, the $i$-th generalized Hamming weight of $\Co$ is 
\begin{equation*}
d_i(\Co)=\min\{\wa{D}: D \text{ is a vector subspace of $\Co$ where} \dim (D)=i\}. 
\end{equation*}
When $D$ is a vector subspace of $\Co$ of dimension $1$, $\wa{D}$ is the Hamming weight of any codeword that generates $D$. Then, $d_1(\Co)=d(\Co)$. Hence, we denote by $d_1(\Co)$ the minimum Hamming weight of $\Co$.

A codeword $\bm{c}\in\Co\setminus\{\bm{0}\}$ has minimal support if its support does not properly contain the support of any other nonzero codeword. We use $\crc{\Co}$ to denote the set of codewords of minimal support of $\Co$.

\begin{Remark}\label{Remark-independiente}
If $\bm{c},\bm{c}'\in \crc{\Co}$ and $\dim\ev{\bm{c},\bm{c}'}=2$, then $\supp{c}\neq \suppa{\bm{c}'}$. If this were not the case,  we could find $\lambda\in\Fn{q}\setminus\{0\}$ such that $\suppa{\bm{c}-\lambda \bm{c}' }\subsetneq \supp{c}$, where $\bm{0}\neq \bm{c}-\lambda \bm{c}' \in \Co$, contradicting the fact that $\bm{c}\in\crc{\Co}$.
\end{Remark}

Let us now introduce some notation and concepts regarding polynomial rings.

Let $K$ be a field and $R=K[x_1,x_2,\ldots,x_n]$. For $S\subset R$,  without risk of confusion, we denote by $\langle S \rangle $ the ideal generated by $S$. 
A monomial order $\prec$ on the monomials of $R$ is a total order such that, for any monomials $u,v,w$ in $R$,
\begin{multicols}{2}
\begin{enumerate}[(a)]
\item if $u\neq 1$, then $1\prec u$;
\item if $u\prec v$, then $uw\prec vw$.
\end{enumerate}
\end{multicols}

For more details on monomial orders, see, for instance,  \cite{herzog2018binomialx}.

Let $f=\sum_{i=1}^sc_iu_i\in R$ where $f\neq 0$  and $c_i\in K\setminus\{0\}$. The initial monomial of $f$ with respect to $\prec$, is defined as $\inic_{\prec}(f)=\max_{\prec}\{u_1,u_2,\ldots,u_s\}$. Given an ideal $\ev{0}\neq I\subset R$, the initial ideal of $I$ with respect to $\prec$ is $$\inic_{\prec}(I)=\langle\inic_{\prec}(f):0\neq f\in I \rangle.$$

Given a fixed monomial order $\prec$ on $R$, and an ideal $\ev{0}\neq I\subset R$, a set $\G=\{g_1,\ldots,g_s\}\subset I$ is called a Gr{\"o}bner basis of $I$, if $\inic_{\prec}(I)=\langle \inic_{\prec}(g_1),\ldots,\inic_{\prec}(g_s) \rangle$. Moreover, $\G$ is called the reduced Gr{\"o}bner basis of $I$ with respect to $\prec$ if  $g_i$ is monic for all $i\in[s]$, and, for $i,j\in[s]$ where $i\neq j$, none of the monomials appearing in $g_j$ is divisible by $\inic_{\prec}(g_i)$. It is a well-known fact that the reduced Gr{\"o}bner basis of $I$ for a fixed monomial order is unique. 

If $f\in R$, there exists a unique $g\in R$ such that $f-g\in I$ and no monomial in the expansion of $g$ belongs to $\inic_{\prec}(I)$. In this case, we call $g$, the canonical form of $f$ with respect to $I$ and $\prec$, and it is denoted by $g=\Canf{f}{I}$.

We omit the reference to the order $\prec$ when it is clear from the context.

\subsection{A binomial ideal associated to a linear code}\label{s:pre-ideal}
In this subsection we discuss a binomial ideal that can be associated to a linear code $\Co\subset\F{q}{n}$. This ideal is introduced in \citet{marquez-corbella_ideal_2016}.

Let $\Co\subset\F{q}{n}$ be a linear code and $\Fn{q}=\{0,\alpha,\alpha^2,\ldots,\alpha^{q-1}=1\}$, where $\alpha$ is a fixed primitive element of $\Fn{q}$. Let $\bm{x}$ denote vector indeterminates $X_1,X_2,\ldots,X_n$, where each $X_i$ can be decomposed in $q-1$ subindeterminates $x_{i,1},\ldots,x_{i,q-1}$. Hence, in this context, $\bm{x}$ denotes the tuple of indeterminates $\bm{x} = (x_{1,1},\ldots,x_{1,q-1}, \dots, x_{n,1},\ldots,x_{n,q-1}).$

Let $\Kx$  be the ring of polynomials over a field $K$, defined as
\begin{equation}
\Kx=K[X_1,X_2,\ldots,X_n]=K[x_{i,j}:1\leq i\leq n,\  1\leq j\leq q-1].\label{eq:ring-kx}
\end{equation}  

Let $E_{q}=\{e_1,\ldots,e_{q-1}\}$ be the canonical basis of $\mathbb{N}^{q-1}$, and let $E$ be the set
$$E=\{(u_1,u_2,\ldots,u_n)\mid u_i \in E_q\cup\{\bm{0}\} \}\subset \mathbb{N}^{n(q-1)},\quad \text{where } \bm{0}\in \mathbb{N}^{q-1}.$$

Consider the bijective function
\begin{align*}
\Delta&:\F{q}{n}&&\to &&E\\
(a_1,a_2,&\ldots,a_n)&&\mapsto &&(u_1,u_2,\ldots,u_n), \quad \text{where }u_i=\begin{cases} e_{j_i}, & a_i=\alpha^{j_i}\\ \bm{0}, & a_i=0 \end{cases}.
 \end{align*}

Finally, consider the binomial ideal
\begin{equation}
\Ic=\langle \xp{a}-\xp{b}\mid \bm{a}-\bm{b}\in\Co\rangle.
\label{eq:ideal-generalizado}
\end{equation}

Given $u,v\in [q-1]$ and $\alpha^u,\alpha^v\in\Fn{q}$, we can have $\alpha^u+\alpha^v=\alpha^w $ for some $w\in [q-1]$, or $\alpha^u+\alpha^v=0 $. For each $i\in [n]$, define the ideals: $$\mathcal{R}_{X_i}=\langle x_{i,u}x_{i,v}-x_{i,w}\mid \alpha^u+\alpha^v=\alpha^w \rangle+\langle  x_{i,u}x_{i,v}-1\mid \alpha^u+\alpha^v=0\rangle.$$

Let $\Rx=\sum_{i=1}^n \mathcal{R}_{X_i}\subset\Kx$. $\Rx$ is an ideal contained in $\Ic$ \cite[Theorem 2.3]{marquez-corbella_ideal_2016}.

\begin{Proposition}\cite[Section 2]{marquez-corbella_ideal_2016}.
For  $\bm{a},\bm{b}\in\F{q}{n}$, it follows that:
\begin{enumerate}
\item\label{p1-aa} $\deg{(\xp{a})}=\w{a}$ (using the standard graduation in $\Kx$).  
\item If $\supp{a}\cap\supp{b}=\emptyset$,  then $\xp{(a+b)}=\xp{a}\cdot\xp{b}$. 
\item $\xp{(a+b)}=\xp{a}\cdot \xp{b} \mod \Rx$. 
\end{enumerate}    
\end{Proposition}

Let us now discuss some properties regarding reduced Gr\"obner bases of $\Ic$.

Let $\prec$ be a monomial order on the monomials of $\Kx$. We consider the order in the indeterminates 
\begin{equation}
X_n\prec X_{n-1}\prec \ldots \prec X_2\prec X_1, \label{orden-indeterminadas}
\end{equation}
this means that for any $i\in[n-1]$ and $j,\ell\in[q-1]$, we have $x_{i+1,\ell}\prec x_{i,j}$.

Given $i\in[n]$, the order between the subindeterminates $x_{i,j}$ where  $j\in[q-1]$ can be chosen arbitrarily, but, for the sake of clarity, we shall choose the order given by $x_{i,q-1}\prec x_{i,q-2}\prec \ldots \prec x_{i,2}\prec x_{i,1}$.

\begin{Proposition}\label{prop-division-binomios-grobner}\cite[Proposition 3]
{marquez-corbella_ideal_2016}.  
Let $\Co\subset \F{q}{n}$ be a linear code,  let $\prec$ be a  degree compatible order  on the monomials of $\Kx$, and let $\G$ be the reduced Gr\"obner basis of $\Ic$ with respect to $\prec$. If $f\in \G\setminus \Rx$, there exist $\bm{a},\bm{b}\in\F{q}{n}$ such that $f=\xp{a}-\xp{b}$ where $\xp{b}\prec \xp{a}$ and $\w{b}\leq \w{a}\leq \w{b}+1$.
\end{Proposition}

\subsection{An order on the elements of $\F{q}{n}$}\label{s:pre-order}

In this subsection, we establish some conventions adopted in the paper.

Throughout the paper, $K$ denotes a field. Given a linear code $\Co\subset\F{q}{n}$, we let $\Kx$ denote the polynomial ring defined in \eqref{eq:ring-kx}.  Unless otherwise stated, we use  the notation introduced in Subsection~\ref{s:pre-ideal} for tuples of indeterminates: $$\bm{x} = (x_{1,1}\ldots,x_{1,q-1}, \dots, x_{n,1},\ldots,x_{n,q-1}).$$

In Subsection~\ref{s:pre-stanley} and Section~\ref{s:free-res}, however, we use the notation $\bm{x} = (x_1, x_2, \ldots, x_n)$.

We consider only degree compatible orders with the standard graduation in $\Kx$. Moreover, we use the order on the indeterminates given by \eqref{orden-indeterminadas}.

Let $\prec$ be an order on $\Kx$, we define an order on the elements of $\F{q}{n}$. Let $\bm{a},\bm{b}\in\F{q}{n}$: 
\begin{equation}
\bm{a}\prec \bm{b} \Leftrightarrow \xp{a}\prec \xp{b}.
\label{eq:orden}
\end{equation}
The order defined in this way on $\F{q}{n}$ is total, since $\Delta$ is a bijective function.

Let $\G$ denote the reduced Gr\"obner basis of $\Ic$, defined in \eqref{eq:ideal-generalizado}.

\begin{Remark}
Let $\bm{a},\bm{b}\in\F{q}{n}$. Since we are using a degree compatible order, if $\bm{a}\prec \bm{b}$, then $\w{a}\leq \w{b}$. Moreover,  $\xp{(a+b)}\preceq \xpp{a}{b}$.
\end{Remark}

For $f=\xp{a}-\xp{b}\in \Ic$ where $\xp{b}\prec\xp{a}$, the codeword associated to $f$ is defined as $\bm{c}_f=\bm{a}-\bm{b}$. 

\begin{Lemma}\label{soportes-rem}
Let $\Co \subset\F{q}{n}$ be a linear code and $\G$ be the reduced Gr{\"o}bner basis of $\Ic$. If $f=\xp{a}-\xp{b}\in \GRx$ where $\xp{b}\prec \xp{a}$, then $\suppa{\bm{c}_f}=\supp{a}\cup \supp{b}$.
\begin{proof}
Let $\bm{a}=(a_1,a_2,\ldots,a_n)$ and $\bm{b}=(b_1,b_2,\ldots,b_n)$. 

The case $\suppa{\bm{c}_f}=\suppa{\bm{a}-\bm{b}}\subset \supp{a}\cup \supp{b}$ is straightforward. For the other inclusion, let $i\in \supp{a}\cup \supp{b}$. Suppose that $i\not\in \suppa{\bm{a}-\bm{b}}$. 

Then, $i\in \supp{a}\cap \supp{b}$. Since $i\not\in \suppa{\bm{a}-\bm{b}}$, it follows that $a_i=b_i=\alpha^j$ for some $j\in[q-1]$. Hence,  $x_{i,j}\mid \xp{a}$ and $x_{i,j}\mid \xp{b}$. Since $\xp{b}\prec\xp{a}$, it follows that
\begin{equation}
\xp{(b-\alpha^j e_i)}=\frac{\xp{b}}{x_{i,j}}\prec \frac{\xp{a}}{x_{i,j}}=\xp{(a-\alpha^j e_i)}. \label{leq:-sup-j}
\end{equation}
Since $(\bm{a}-\alpha^j \bm{e_i})-(\bm{b}-\alpha^j \bm{e_i})=\bm{a}-\bm{b}\in \Co$, we have $\xp{(a-\alpha^j e_i)}-\xp{(b-\alpha^j e_i)}\in \Ic$.  Then, from \eqref{leq:-sup-j} it follows that $\xp{(a-\alpha^j e_i)}\in \inicC$.  However,  $\xp{(a-\alpha^j e_i)}$ divides $\xp{a}$ and $\xp{a}$ is a minimal generator of $\inicC$. Then, $i\in \suppa{\bm{a}-\bm{b}}$.
\end{proof}
\end{Lemma}

\subsection{Minimal free resolutions of rings associated to codes}\label{s:pre-stanley}

In this subsection, we discuss some properties of the minimal free resolution of a square-free monomial ideal associated to the supports of codewords in a code $\Co$.

Let $K$ be a field and let $R=K[x_1,\ldots,x_n]$. In this subsection, and later in Section~\ref{s:free-res}, we use the  notation $\bm{x}=(x_1,x_2,\ldots,x_n)$.  Hence, for $\sigma \subset [n]$ and $\sigma\neq\emptyset$, let $\bm{x}^\sigma$ denote the monomial  $\textstyle\prod_{i\in \sigma}x_i\in R$ and $\bm{x}^{\emptyset}=1$.

Let $I\subset R$ be an ideal. By Hilbert’s Syzygy Theorem, $R/I$ admits a minimal free resolution as a $\mathbb{N}$-graded module
\begin{equation}
0 \longrightarrow F_{\rho} \stackrel{\partial_{\rho}}\longrightarrow \cdots \stackrel{\partial_{3}}\longrightarrow F_2 \stackrel{\partial_{2}}\longrightarrow F_1 \stackrel{\partial_{1}}\longrightarrow F_0=R \stackrel{\partial_{0}}\longrightarrow R/I \longrightarrow 0, \label{res-sr}
\end{equation}
where $
F_i=\textstyle \bigoplus_{j \in \mathbb{N}} R(-j)^{\beta_{i, j}}
$ for $i\in\{0,1,..,\rho\}$. The integer $\rho$ is called the projective dimension of $R/I$ and is denoted by $\pd(R/I)$.
The numbers $\beta_{i, j}$ are called the $\mathbb{N}$-graded Betti numbers of $R/I$ over $K$, and are independent of the choice of the minimal free resolution. 

Let us focus in the monomial ideal associated with the supports of codewords in a code. Given a linear code $\Co(n,k)_q$ and $ M\subset \crc{\Co}$, we denote the set of supports of codewords in $M$ as $S_M$. Hence, $S_M=\{\sigma\subset [n]:\exists \bm{c}\in M\text{ s. t. }\supp{c}=\sigma\}$. We associate to $M$ the monomial ideal
\begin{equation}
\ideal{M}=\langle \bm{x}^{\sigma}\mid \sigma\in S_M  \rangle \subset R. \label{ideal-M-g}
\end{equation}

In \citet{johnsen2013hamming}, it is proved that the Betti numbers of $R/\ideal{\crc{\Co}}$ are independent of the field $K$ and $\pd(R/\ideal{\crc{\Co}})=k$. In addition, the authors prove that the GHWs of $\Co$ can be determined using the Betti numbers of $R/\ideal{\crc{\Co}}$. More precisely, we have the following:

\begin{Teorema}\label{prop:betti-verdure}\cite[Theorem 2]{johnsen2013hamming}.
Let $\Co(n,k)_q$ be a linear code. Then 
\begin{equation}
d_i(\mathcal{C})=\min \{j \mid \beta_{i,j}(R/\ideal{\crc{\Co}})\neq 0\}, \qquad i\in [k].\label{beti-verdure-teo}
\end{equation}
    
\end{Teorema}

\section{On $d_2$-test sets}\label{s:d2-test-set}

The concept of a test set was first introduced in \citet{goos_minimun_1997}, and has been studied by several authors \cite{borges-quintana_computing_2015,borges2008grobner,pinto_fan_2015,helleseth_error-correction_2005,marquez-corbella_ideal_2016}. In \cite{garcia2022free} it is defined a $\G_{\prec}$-test for binary codes $\Co$, as a subset of $\crc{\Co}$ where we can find a codeword of minimum Hamming weight $d_1(\Co)$. 

We can generalize this concept for $d_2(\Co)$ by introducing the notion of a $d_2$-test set.

\begin{Definicion}[$d_2$-test set]
Given a linear code $\Co\subset\F{q}{n}$, we say that a set $M\subset \crc{\Co}$ is a $d_2$-test set of $\Co$, if there exist $\bm{c}_1,\bm{c}_2\in M$ such that $\dim\ev{\bm{c}_1,\bm{c}_2}=2$, and $\wev{\bm{c}_1,\bm{c}_2}=d_2(\Co)$.
\end{Definicion}

Let $\Co\subset\F{q}{n}$ be a linear code and consider the set $\mathcal{T}$ consisting of the nonzero codewords associated with the reduced Gr\"obner basis $\G$ of $\Ic$, i.e. 
$$\mathcal{T}=\{\bm{c}\in \Co \mid \exists g\in\GRx\text{ s. t. } \bm{c}_g=\bm{c}\}.$$ 

\begin{Remark}
In the binary case $\mathcal{T}\subset\crc{\Co}$ \cite{borges2008grobner}. Moreover, it is a $d_2$-test set \cite[Theorem 2]{garcia2022free}. However, if $\Co\subset\F{q}{n}$,  $\mathcal{T}$ is not always a  subset of $\crc{\Co}$ \cite[Section 4]{marquez-corbella_ideal_2016}.    
\end{Remark}
In order to deal with the issues mentioned  in  the previous Remark, let us define:
\begin{equation}
M_\G=\{\bm{c}\in \Co \mid \exists g\in\GRx\text{ s. t. } \bm{c}_g=\bm{c}\}\cap \crc{\Co}.   \label{def-Mg}
\end{equation}

The set $M_\G$ contains a codeword of minimum Hamming weight. More precisely, we have the following:

\begin{Proposition}\label{prop:d1-Mg}\cite[Section 3]{marquez-corbella_ideal_2016}. 
Let $\Co\subset\F{q}{n}$ be a linear code and let $\G$ be the reduced Gr\"obner basis of $\Ic$. Then, there exists $\bm{c}\in M_\G$ such that $\w{c}=d_1(\Co)$.
\end{Proposition}

Let us introduce some sets that play a key role in determining when $M_\G$ is a $d_2$-test set. These sets generalize similar sets introduced in the binary case in \cite{garcia2022free}.

For a linear code $\Co(n,k)_q$ where $k\geq 2$, denote by $\bm{M}_{1}$ the set:  $$\bm{M}_1=\left\{\bm{m} \in \mathcal{C}\setminus\{\bm{0}\} \mid \exists \bm{m}' \in \Co \text { such that } d_2(\Co)=\wev{\bm{m},\bm{m}'}\right\}. $$
Using the order in \eqref{eq:orden}, define $\mun=\min_{\prec} (\bm{M}_1)$, and the set $$\bm{M}_2=\left\{\bm{m} \in\Co \mid d_{\ii}(\mathcal{C})=\wev{\mun, \bm{m}}\right\}.$$
Now, let $\bm{m}_2=\min_{\prec}\bm{M}_2$. Finally, denote 
$\I =\suppa{\mun}$ and $\J =\suppa{\md}$.

By the definitions of $\mun$ and $\md$, we have
\begin{equation}
d_2(\Co)=\wev{\bm{m}_1,\bm{m}_2}=|\suppa{\bm{m}_1}\cup\suppa{\bm{m}_2}|=|\I\cup \J |.
\label{eq:d2-equi}
\end{equation}

When an explicit reference to the linear code $\Co$ is needed, we write $\mun(\Co)$ and $\md(\Co)$ in place of $\mun$ and $\md$. Analogously, we write $\bm{M}_1(\Co)$ and $\bm{M}_2(\Co)$ for the sets
$\bm{M}_1$ and $\bm{M}_2$.

\begin{Lemma}\label{lemma-m1-d2}
Let $\Co\subset\F{q}{n}$ be a linear code. We have  $|\I|=\w{\mun}<d_2(\Co)$. In particular $\mun\in \crc{\Co}$.
\begin{proof} Because of the definition of $\mun$, there exists $\bm{m}\in\Co$ such that $d_2(\Co)=\wev{\mun,\bm{m}}$. Hence
$$d_2(\Co)=|\supp{m}\cup\suppa{\mun}|\geq |\suppa{\mun}|=\wa{\mun}.$$

Let us assume, for the sake of contradiction, that $d_2(\Co)=\wa{\mun}$. Let $\bm{c}_1,\bm{c}_2\in\Co$ such that $\dim\ev{\bm{c}_1,\bm{c}_2}=2$ and $d_2(\Co)=\wev{\bm{c}_1,\bm{c}_2}$. Without loss of generality, let $\bm{c}_1\prec\bm{c}_2$. Since $\bm{c}_1\in \bm{M}_1$, and $\mun=\min_{\prec}\bm{M}_1$, it follows that $\mun\preceq\bm{c}_1$. Hence, $\wa{\mun}\leq \wa{\bm{c}_1}$. Therefore $$d_2(\Co)=|\suppa{\bm{c}_1}\cup\suppa{\bm{c}_2}|\geq |\suppa{\bm{c}_1}|=\wa{\bm{c}_1}\geq \wa{\mun}=d_2(\Co).$$
Then, $d_2(\Co)=|\suppa{\bm{c}_1}\cup \suppa{\bm{c}_2}|=|\suppa{\bm{c}_1}|=\w{\bm{c}_1}=\wa{\mun}$. 
Hence, $\suppa{\bm{c}_2}\subset \suppa{\bm{c}_1}$. Since $\bm{c}_1\prec \bm{c}_2$, it follows that $\wa{\bm{c}_1} \leq \wa{\bm{c}_2}$. Thus, $\suppa{\bm{c}_2}= \suppa{\bm{c}_1}$.  Since $\dim \ev{\bm{c}_1,\bm{c}_2}=2$, one can choose $\lambda  \in \Fn{q}\setminus\{0\}$ such that $\bm{c}=\bm{c}_1-\lambda \bm{c}_2\neq 0$ and $\supp{c}\subsetneq \suppa{\bm{c}_1}$. Hence
\begin{equation}
0<\w{c}<\wa{\bm{c}_1}. \label{a1n}
\end{equation}

On the other hand, since $\ev{\bm{c}_1,\bm{c}_2}=\ev{\bm{c}_1,\bm{c}}$, $\bm{c}\in \bm{M}_1$. Hence, $\mun\preceq\bm{c}$. Then, $\wa{\bm{c}_1}=\wa{\mun}\leq \w{c}$, which  contradicts \eqref{a1n}. Therefore, $\wa{\mun}<d_2(\Co)$.

Now, if $\mun\not \in\crc{\Co}$, there exists $\bm{0}\neq \bm{c}\in\Co$  that $\supp{c}\subsetneq \suppa{\mun}$. Then, $\dim(\ev{\mun,\bm{c}})=2$. Hence, $\wev{\mun,\bm{c}}\geq d_2(\Co)$. However
$$\wev{\mun,\bm{c}}=|\suppa{\mun}\cup\supp{c}|=|\suppa{\mun}|=\wa{\mun}<d_2(\Co).$$
Thus, $\mun\in\crc{\Co}$.
\end{proof}
\end{Lemma}

\begin{Corollary}
Let $\Co\subset \F{q}{n}$ be a linear code. Then, $\bm{0}\not\in \bm{M}_{2}$ and $\I \neq \J $. In particular, $\dim \ev{\mun,\md}=2$ and $\bm{M}_{2}\subset \bm{M}_{1}$. As a result $\mun\prec \md$, and therefore
\begin{equation}
|\I|\leq |\J |. \label{m1Mm2}
\end{equation}    
\end{Corollary}

\begin{Remark}\label{re:M1-M2-dim2}
For a linear code $\Co=\Co(n,2)_q$, we have $\bm{M}_1=\Co\setminus\{0\}$ and $\bm{M}_2=\Co\setminus \ev{\mun}$.
\end{Remark}

\begin{Lemma} Let $\Co\subset\F{q}{n}$ be a linear code. We have  $|\J|=\w{\md}<d_2(\Co)$.

\begin{proof}  
Since $d_2(\Co)=|\I \cup \J|$, it follows that $d_2(\Co)\geq |\J|=\wa{\md}$. The condition $d_2(\Co)=|\J|$ holds if and only if $\I \subset \J$. Therefore, it is enough to show that $\I \not\subset \J$.

We assume for the sake of contradiction that $\I \subset \J$. Since $\dim\ev{\mun,\md}=2$ and $\I \subset \J$, we can take  $\bm{m}=\md-\lambda\mun$ where $\lambda \in\Fn{q}\setminus \{0\}$, such that $\supp{m}\subsetneq  \suppa{\md}$. Since $\ev{ \mun, \md} = \ev{ \mun,\bm{m}}$, $d_2(\Co)=\wev{\mun, \bm{m}}$. Then, $\bm{m}\in \bm{M}_2$. Hence, $\md=\min_{\prec}\bm{M}_2\preceq \bm{m}$, but $\w{m}<\w{\md}$. Thus, $\I \not\subset \J$. 
\end{proof}
\end{Lemma}

\begin{Lemma}\label{lemma-cg-m2}
Let $\Co\subset\F{q}{n}$ be a linear code. If $\bm{m}\in   \bm{M}_2$ and $\w{m}=\wa{\md}$, then $\bm{m}\in\crc{\Co}$. In particular $\md\in\crc{\Co}$.

\begin{proof}
Suppose $\bm{m}\not\in\crc{\Co}$. Then, there exists $\bm{c}\in\Co$ such that $\supp{c}\subsetneq \supp{m}$. Hence 
\begin{equation}
\w{c}<\w{m}=\wa{\md}. \label{lemma-001}
\end{equation}
First, we prove that $\bm{c}\in\ev{\mun}$. Suppose $\dim\ev{\mun,\bm{c}}=2$. Then, $\wev{\mun,\bm{c}}\geq d_2(\Co)$. Moreover
$$\wev{\mun,\bm{c}}=|\I \cup\supp{c}|\leq |\I \cup\supp{m}|=\wev{\mun,\bm{m}}=d_2(\Co).$$
Thus, $\wev{\mun,\bm{c}}=d_2(\Co)$. Hence, $\bm{c}\in\bm{M}_2$. Then, $\md=\min_{\prec}\bm{M}_2\preceq \bm{c}$, contradicting \eqref{lemma-001}. Therefore, $\bm{c}\in\ev{\mun}$. Hence
\begin{align*}
d_2(\Co)=\wev{\mun,\bm{m}}=\wev{\bm{c},\bm{m}}&=|\supp{c}\cup\supp{m}|\\
&=|\supp{m}|=\w{m}=\wa{\md}<d_2(\Co).
\end{align*}
Therefore, $\bm{m}\in\crc{\Co}$.
\end{proof}
\end{Lemma}

\begin{Corollary}\label{lemma-m1-m2-d2-test-set}
Let $\Co\subset\F{q}{n}$ be a linear code and $M\subset \crc{\Co}$. If $\mun\in M$, and there exists $\bm{m}\in \bm{M}_2\cap M$ such that $\w{m}=\wa{\md}$, then $M$ is a $d_2$-test set. 
\end{Corollary}

\section{A sufficient condition for $M_\G$ being a \texorpdfstring{$d_2$}{TEXT}-test set }\label{s:mg-si-d2}

In this section we prove our first main theorem, which provides  sufficient conditions for the set $M_\G$, defined in \eqref{def-Mg}, to be a $d_2$-test set (see Theorem~\ref{corollary-main}).

According to Corollary~\ref{lemma-m1-m2-d2-test-set}, given a linear code $\Co\subset\F{q}{n}$, in order for a set $M\subset \crc{\Co}$ to be a  $d_2$-test set, it suffices that $\mun\in M$ and that there exists a codeword  $\bm{m}\in \bm{M}_2\cap M$ such that $\w{m}=\wa{\md}$. Accordingly, in this section we derive conditions under which these requirements are satisfied for $M_\G$.

If $\Co\subset\F{2}{n}$, we have the following \cite[Remark 3]{garcia2022free}: 
\begin{equation}
|\I \cap \J |\leq \frac{|\I |}{2}\leq \frac{|\J |}{2}. \label{eq:desi-binary-case}
\end{equation}
These inequalities represent a crucial ingredient to prove that $M_\G$ is a $d_2$-test set in the binary case. However, this result does not hold in general for linear codes $\Co\subset \F{q}{n}$, where $q>2$.

\begin{Ejemplo}
Let $\Co\subset \F{3}{8}$ be the linear code with generator matrix given by
$$G=\begin{pmatrix}
1&1&1&2&1&2&0&0\\
0&0&1&1&1&1&1&1
\end{pmatrix}.$$
The linear code has $3^2=9$ codewords, each of which has Hamming weight $6$. Then, for any degree compatible order we have $|\I|=|\J|=6$. Also $$d_2(\Co)=\wa{\Co}=\wev{(1,1,1,2,1,2,0,0),(0,0,1,1,1,1,1,1)}=8.$$ Then, $|\I\cap\J|=4$. Hence, we have  $4=|\I \cap \J |>\tfrac{|\I |}{2}=3$.
\end{Ejemplo}

In general, we have the following:

\begin{Proposition}\label{lema-intersectionIJ}
Given a linear code $\Co\subset\F{q}{n}$, we have 
\begin{equation}
|\I \cap \J |\leq \frac{q-1}{q}|\I |\leq \frac{q-1}{q}|\J |. \label{eq:total-ineq-int}
\end{equation}
\begin{proof}
The last inequality of \eqref{eq:total-ineq-int} follows from \eqref{m1Mm2}.

Now, for each $ i\in[q-1]$, we define $E_i=\{j\in \I \cap \J:  j\notin \suppa{\mun-\alpha^i \md} \}$.
Then
\begin{equation}
\I \cap \J=\bigsqcup_{i=1}^{q-1} E_i. \label{eq:car-inter}
\end{equation}

Let $r\in[q-1]$ be such that $|E_r|=\max\{|E_i|:i\in[q-1]\}$. From \eqref{eq:car-inter}, we can observe that $|E_r|\geq \pecp{|\I \cap \J|}{q-1}$. If the equality holds, then $\pecp{|\I \cap \J|}{q-1}=\tfrac{|\I \cap \J|}{q-1}$.  As a consequence, if $(q-1)$ does not divide $ |\I \cap \J|$, then $|E_r|\geq \pecp{|\I \cap \J|}{q-1}+1$.

On the other hand, we know that
\begin{align}
\wa{\mun-\alpha^i \md}&=  |\I \setminus \J|+|\J\setminus\I |+|\{j\in \I \cap \J:j\in \suppa{\mun-\alpha^i \md}\}|  \nonumber\\
&=|\I |-|\I \cap \J|+|\J|-|\I \cap \J|+|\I \cap \J|-|E_i|\nonumber\\
&=|\I |+|\J|-|\I \cap \J|-|E_i|.   \label{eq:we-diferencia}
\end{align}

Moreover, from the definition of $\md$, for each $i\in [q-1]$, $\md\prec\mun-\alpha^i \md$. Then, $\wa{\md}\leq \wa{\mun-\alpha^i \md}$. Hence, from \eqref{eq:we-diferencia} it follows that $|\J|\leq |\I |+|\J|-|\I \cap \J|-|E_i|$. Then, $|\I \cap \J|\leq |\I |-|E_i|$ for all $i\in[q-1]$. In particular
\begin{equation}
|\I \cap \J|\leq |\I |-|E_r|. \label{eq:first-in-int}
\end{equation}

If $|E_r|= \pecp{|\I \cap \J|}{q-1}=\tfrac{|\I \cap \J|}{q-1}$, then $
|\I \cap \J|\leq |\I |-  \tfrac{|\I \cap \J|}{q-1}$. Hence, $|\I \cap \J|\leq \tfrac{q-1}{q}|\I |$.

Now suppose that $|E_r|\geq \pecp{|\I \cap \J|}{q-1}+1$. Since $\pecp{|\I \cap \J|}{q-1}> \tfrac{|\I \cap \J|}{q-1}-1$, it follows that $|E_r|>\tfrac{|\I \cap \J|}{q-1}$. Now, using \eqref{eq:first-in-int}, we have $
|\I \cap \J|\leq |\I |-|E_r| < |\I |-\tfrac{|\I \cap \J|}{q-1}$. Then
\begin{equation*}
|\I \cap \J|< \frac{q-1}{q}|\I |. \qedhere    \label{eq:second-eq-ine} 
\end{equation*}
\end{proof}
\end{Proposition}

The following lemma is a technical result that we need:

\begin{Lemma}\label{obs-01}
Let $\bm{a},\bm{b},\bm{v}\in \F{q}{n}\setminus\{\bm{0}\}$ be three different words where $\supp{a}\cap\supp{b}=\emptyset$ and $\w{v}\leq\w{a}$. Then, $\dim\ev{\bm{a}-\bm{b},\bm{a}-\bm{v}}=2$.
\begin{proof}
Suppose $\bm{a}-\bm{b}\in\ev{\bm{a}-\bm{v}}$. 
Since $\bm{b}\neq\bm{v}$, there exists $\lambda\in\Fn{q}\setminus\{0,1\}$ such that $\bm{a}-\bm{b}=\lambda(\bm{a}-\bm{v})$. Hence, $(1-\lambda)\bm{a}=\bm{b}-\lambda\bm{v}$. Then,
$\supp{a}=\suppa{\bm{b}-\lambda\bm{v}}\subset \supp{b}\cup\supp{v}$.

Since $\supp{a}\cap\supp{b}=\emptyset$, $\supp{a}\subset \supp{v}$. Since $\w{v}\leq \w{a}$,  $\supp{a}=\supp{v}$. It follows that $\suppa{\bm{a}-\bm{v}}\subset \supp{a}$. Therefore
$$\supp{a}\cup\supp{b}=\suppa{\bm{a}-\bm{b}}=\supp{\bm{a}-\bm{v}}\subset \supp{a}.$$
This is a contradiction because $\supp{a}\cap\supp{b}=\emptyset$. Thus, $\dim\ev{\bm{a}-\bm{b},\bm{a}-\bm{v}}=2$.
\end{proof}
\end{Lemma}

In the binary case, one can find binomials $f,g\in\G$, such that $\bm{c}_f=\mun$ and $\bm{c}_g=\md$ \cite[Section 4]{garcia2022free}. In the following propositions, we generalize this result for linear codes in $\F{q}{n}$, adding a condition on the size of $\I \cap \J$. With such condition, the proof of these propositions can be adapted from the corresponding results in the binary case.

\begin{Proposition}\label{prop-m1}
Let $\Co\subset\F{q}{n}$ be a linear code where $|\I \cap \J|\leq \frac{|\J|+1}{2}$, and let $\G$ be the reduced Gr\"obner basis of $\Ic$. Then, there exists $f\in \G$ such that $\bm{c}_f=\bm{m}_1$.
\begin{proof}
We write $\mun$ as $\mun=\bm{a}-\bm{b}$, where $\supp{a}\cap\supp{b}=\emptyset$, $\w{b}\leq\w{a}\leq \w{b}+1$, and $\xp{b}\prec\xp{a}$.

Let $f=\xp{a}-\xp{b}\in\Ic$. Since $\xp{b}\prec\xp{a}$, we have $\xp{a}\in\inicC$. Then, there exists $h=\xp{u}-\xp{v}\in\GRx$ where $\xp{v}\prec\xp{u}$ and $\xp{u}\mid \xp{a}$. We show that $\bm{u}=\bm{a}$ and $\bm{v}=\bm{b}$.

Suppose $\bm{u}\neq \bm{a}$. Then, $\w{u}\leq \w{a}-1$, so
$\w{v}\leq\w{u}\leq \w{a}-1\leq   \w{b}$.
It follows that $\wa{\bm{u}-\bm{v}}\leq \w{u}+\w{v}\leq \w{a}-1+\w{b}=\w{\mun}-1<\w{\mun}$. Then, $\dim\ev{\mun,\bm{u}-\bm{v}}=2$. So, $\wev{\mun,\bm{u}-\bm{v}}\geq d_2$. Moreover, since $\wa{\bm{u}-\bm{v}}<\w{\mun}$, we have $\bm{u}-\bm{v}\prec\mun\prec\md =\min_{\prec}\bm{M}_2$. This implies that $\bm{u}-\bm{v}\notin \bm{M}_2$. Hence, $\wev{\mun,\bm{u}-\bm{v}}>d_2(\Co)$. As a result 
\begin{align*}
|\I \cup \J|=  d_2(\Co)&\leq \wev{\mun,\bm{u}-\bm{v}}-1\\
&=|\supp{\mun}\cup \suppa{\bm{u}-\bm{v}}|-1\nonumber\\
&=|\supp{a}\cup \supp{b}\cup \supp{u}\cup \supp{v}|-1\nonumber \quad \text{(by Lemma~\ref{soportes-rem})}\\
&=|\supp{a}\cup \supp{b}\cup \supp{v}|-1 \quad (\text{since }\supp{u}\subset \supp{a})\nonumber\\
&=|\I \cup \supp{v}|-1\\ 
&\leq |\I |+\w{v}-1 \\
&\leq |\I |+\w{u}-1\qquad (\text{since $\xp{v}\prec \xp{u}$}).
\end{align*}    
It then follows that
$ |\I |+|\J|-|\I \cap \J|=|\I \cup \J|\leq |\I |+\w{u}-1$. Hence
\begin{equation}
|\J|-|\I \cap \J|\leq \w{u}-1.  \label{otra-1}
\end{equation}
On the other hand, since $|\I \cap \J|\leq  \frac{|\J|+1}{2}$, we also have
\begin{align}
|\J|-|\I \cap \J|\geq \frac{|\J|-1}{2}\geq \frac{|\I |-1}{2}&=\frac{\w{a}+\w{b}-1}{2}\nonumber\\
&\geq \frac{\w{a}+\w{a}-1-1}{2}=\w{a}-1. \label{otra-2}
\end{align}

Combining \eqref{otra-1} and  \eqref{otra-2} we obtain:
$\w{a}-1\leq |\J|-|\I \cap \J|\leq \w{u}-1$. Then, $\w{a}\leq \w{u}$. However, by assumption $\w{u}\leq \w{a}-1$. Hence,  $\bm{u}=\bm{a}$.

\medskip

Suppose now that $\bm{v}\neq \bm{b}$. Since $\xp{a}-\xp{b},\xp{a}-\xp{v}\in \Ic$,  $\xp{b}-\xp{v}\in \Ic$. 
Given that $\xp{v}\not\in\inicC$, we have $\xp{v}\prec \xp{b}\prec\xp{a}$. Then, by Lemma~\ref{obs-01},  $\dim\ev{\bm{a}-\bm{b},\bm{a}-\bm{v}}=2$, or equivalently $\dim\ev{\bm{a}-\bm{b},\bm{v}-\bm{b}}=2$. 

Since $\xp{v}\prec\xp{a}$, $\xpp{(-b)}{v}\prec\xpp{(-b)}{a}$. Moreover, $\xp{(v-b)}\preceq \xpp{(-b)}{v}$ and $\xpp{(-b)}{a}=\xp{\mun}$. Then, $\xp{(v-b)}\prec\xp{\mun}$.  Consequently, $\bm{v}-\bm{b}\prec\mun\prec\md=\min_{\prec}\bm{M}_2$.  Then, $\bm{v}-\bm{b}\notin \bm{M}_2$. Therefore, $\wev{\mun,\bm{v}-\bm{b}}>d_2(\Co)$. It follows that
\begin{align*}
|\I \cup \J|=d_2(\Co)&\leq \wev{\mun,\bm{v}-\bm{b}}-1\\
&=|\supp{\mun}\cup \suppa{\bm{v}-\bm{b}}|-1\nonumber\\
&\leq |\supp{a}\cup \supp{b}\cup \supp{v}\cup \supp{b}|-1\nonumber \\
&=|\supp{a}\cup \supp{b}\cup \supp{v}|-1 \nonumber\\
&=|\I \cup \supp{v}|-1\\  
&\leq |\I |+\w{v}-1 \\
&\leq |\I |+\w{b}-1 \qquad (\text{since $\xp{v}\prec \xp{b}$}).
\end{align*}
Then, $ |\I |+|\J|-|\I \cap \J|=|\I \cup \J|\leq |\I |+\w{b}-1$. It follows that
\begin{equation}
|\J|-|\I \cap \J|\leq \w{b}-1.  \label{otra-3}
\end{equation}
Additionally, since $|\I \cap \J|\leq  \frac{|\J|+1}{2}$, we have
\begin{align}
|\J|-|\I \cap \J|\geq \frac{|\J|-1}{2}\geq \frac{|\I |-1}{2}&=\frac{\w{a}+\w{b}-1}{2}\nonumber\\
&\geq \frac{\w{b}+\w{b}-1}{2}=\w{b}-\frac{1}{2}. \label{otra-4}
\end{align}
By combining inequalities \eqref{otra-3} and \eqref{otra-4}, we obtain $\w{b}-\frac{1}{2}\leq |\J|-|\I\cap \J|\leq \w{b}-1$, which is a contradiction. Hence,  $\bm{v}=\bm{b}$.

It follows that $f=h \in \GRx$, where $\bm{c}_f=\mun$.
\end{proof}
\end{Proposition}

For Proposition~\ref{prop-m2} and Theorem~\ref{corollary-main}, we consider degree compatible orders that satisfy the following condition: 
For $\bm{a},\bm{b}\in\F{q}{n}$ where $\supp{a}\cap\supp{b}=\emptyset$,
\begin{equation}
\text{ if $\xp{b}\prec \xp{a}$, then $\xp{(-b)}\prec \xp{(-a)}$.}\label{condicion-menos-orden}
\end{equation}

\begin{Remark} The \textit{lexicographic order} $\prec_{\text{lex}}$ and the \textit{reverse lexicographic order} $\prec_{\text{rev}}$, as defined in \cite{herzog2018binomialx}, are examples of orders that satisfy the condition in \eqref{condicion-menos-orden}. Moreover,  when  
the field $\Fn{q}$ has characteristic two, i.e., $q=2^s$, the condition in \eqref{condicion-menos-orden} is satisfied for every order.
\end{Remark}

\begin{Proposition}\label{prop-m2}
Let $\Co\subset\F{q}{n}$ be a linear code where $|\I \cap \J|\leq \tfrac{|\J|+1}{2}$, and let $\G$ be the reduced Gr\"obner basis of $\Ic$. Then, there exists  $g\in \G$ such that $\wev{\bm{m}_1,\bm{c}_g}=d_2(\Co)$ and $\wa{\bm{c}_g}=\wa{\md}$. Moreover, if $q=2^s$, there exists  $g\in \G$ such that  $\bm{c}_g=\md$.
\begin{proof}
Since $2|\I \cap \J|\leq |\J|+1$, it follows that $|\I \cap \J|\leq |\J|-|\I \cap \J|+1=|\J\setminus \I |+1$. Now we can write $\md$ as $\md=\bm{a}+\bm{b}_1-\bm{b}_2$ where   $\supp{a}=\I \cap \J$, $\suppa{\bm{b}_1}\sqcup \suppa{\bm{b}_2}=\J\setminus\I $, $\suppa{\bm{a}+\bm{b}_1}\cap \suppa{\bm{b}_2}=\emptyset$, and $\wa{\bm{b}_2}\leq \wa{\bm{a}+\bm{b}_1}\leq \wa{\bm{b}_2}+1$. Notice that $\supp{a}\cap \suppa{\bm{b}_1}=\emptyset$.

\medskip

\noindent \textbf{Case 1}. $\xp{(a+b_1)}\succ \xp{b_2}$.

Let $f=\xp{(a+b_1)}-\xp{b_2}\in\Ic$. Then, $\xp{(a+b_1)}\in\inicC$, so there exists $g=\xp{u}-\xp{v}\in\GRx$ such that $\xp{v}\prec\xp{u}$ and $\xp{u}\mid\xp{(a+b_1)}$.

We begin by showing that $\dim \ev{\bm{u}-\bm{v},\mun}=2$.
\begin{itemize}
\item   Suppose $\bm{u}-\bm{v} \in\ev{\mun}$. We show that this assumption implies $\supp{u}\cap\suppa{\bm{b}_1}=\emptyset$.  Indeed, if $i\in \supp{u}$,  since $\xp{u}-\xp{v}\in\G$,  $i\in\supp{u}\cup\supp{v}=\suppa{\bm{u}-\bm{v}}=\suppa{\mun}=\I $. On the other hand, since $\suppa{\bm{b}_1}\subset \J\setminus\I $, $i\notin \suppa{\bm{b}_1}$. Therefore, $\supp{u}\cap\suppa{\bm{b}_1}=\emptyset$.

Since $\xp{u}\mid \xpp{a}{b_1}$, it follows that $\xp{u}\mid \xp{a}$. Thus, $\xpp{(a-u)}{u}=\xp{a}$.

On the other hand, since $\bm{u}-\bm{v} \in\ev{\mun}$, we have  $\ev{\mun,\md}=\ev{\mun,\md-(\bm{u} - \bm{v})}$. This implies that $\md-(\bm{u} - \bm{v})\in\bm{M}_2$. Since $\md=\min_{\prec}\bm{M}_2$, it follows that
\begin{equation}
\md\prec\md-(\bm{u} - \bm{v}).
\label{eq:des-p2aa}
\end{equation}

On the other hand, from $\xp{v}\prec\xp{u}$ we obtain
$$\xp{(a-u+v)}\preceq \xpp{(a-u)}{v}\prec\xpp{(a-u)}{u}=\xp{a}.$$

It follows that 
$$\xp{(a+\bm{b}_1-\bm{b}_2-(u-v))}\preceq \xpp{(\bm{b}_1-\bm{b}_2)}{(a-u+v)}\prec\xpp{(\bm{b}_1-\bm{b}_2)}{a}=\xp{\md}.$$
Therefore, $\md-(\bm{u}-\bm{v})=\bm{a}+\bm{b}_1-\bm{b}_2-(\bm{u}-\bm{v})\prec\md$, which contradicts \eqref{eq:des-p2aa}. Thus, we conclude  $\dev{\mun,\bm{u}-\bm{v}}=2$. Hence, $\wev{\mun,\bm{u}-\bm{v}}\geq d_2(\Co)$.
\end{itemize}

Now we observe that
\begin{align}
\wev{\mun,\bm{u}-\bm{v}}&=| \suppa{\mun}\cup \suppa{\bm{u}-\bm{v}}| \nonumber\\
&=|\I \cup\supp{u}\cup \supp{v}| \nonumber\\
&\leq  |\I \cup \supp{u}|+\w{v} \nonumber\\
&\leq |\I \cup \suppa{\bm{b}_1}|+\w{v}\qquad (\text{since $\I \cup\supp{u}\subset \I \cup\suppa{\bm{b}_1}$}) \nonumber\\
&= |\I |+\wa{\bm{b}_1}+\w{v}.\label{prop-oc01}
\end{align}

Now we prove that $\bm{u}= \bm{a}+\bm{b}_1$.
\begin{itemize}
\item Assume $\bm{u}\neq \bm{a}+\bm{b}_1$. Then, $\w{u}< \wa{\bm{a}+\bm{b}_1}$. Thus
\begin{align}
\w{v}\leq \w{u} \leq \wa{\bm{a}+\bm{b}_1}-1\leq \wa{\bm{b}_2}. \label{otra:pro-02}
\end{align}
Combining \eqref{prop-oc01} and \eqref{otra:pro-02} we obtain
\begin{align*}
\wev{\mun,\bm{u}-\bm{v}}\leq |\I |+\wa{\bm{b}_1}+\w{v}
\leq |\I |+\wa{\bm{b}_1}+ \wa{\bm{b}_2}=d_2(\Co).
\end{align*}
Hence, $\wev{\mun,\bm{u}-\bm{v}}=d_2(\Co)$. Then,  $\bm{u}-\bm{v}\in\bm{M}_2$. Since $\md=\min_{\prec}\bm{M}_2$, $\md\preceq \bm{u}-\bm{v}$. Consequently  $\wa{\md}\leq \wa{\bm{u}-\bm{v}}$. However
\begin{align*}
\wa{\bm{u}-\bm{v}}&\leq \w{u}+\w{v}<\wa{\bm{a}+\bm{b}_1}+\w{v}\leq \wa{\bm{a}+\bm{b}_1}+\wa{\bm{b}_2}=\w{\md}. 
\end{align*}
This is a contradiction. Therefore, $\bm{u}= \bm{a}+\bm{b}_1$.
\end{itemize}

Now we prove that $g=\xp{(a+b_1)}-\xp{v}$ satisfies the required condition.
\begin{itemize}
\item If $\bm{v}=\bm{b}_2$, then $g=\xp{(a+b_1)}-\xp{b_2}\in\GRx$ and $\wev{\mun, \bm{c}_g}=\wev{\mun, \md}=d_2(\Co)$.

    \item Assume $\bm{v}\neq\bm{b}_2$. Since both $g=\xp{(a+b_1)}-\xp{v}$ and $f=\xp{(a+b_1)}-\xp{b_2}$ belong to $\Ic$, it follows that $\xp{b_2}-\xp{v}\in \Ic$. Since $\xp{v}\not\in\inicC$, $\xp{v}\prec \xp{b_2}$. Hence, $\w{v}\leq \wa{\bm{b}_2}$.  From \eqref{prop-oc01}, it follows that
\begin{align*}
\wev{\mun, \bm{a}+\bm{b}_1-\bm{v}}&=\wev{\mun, \bm{u}-\bm{v}}\\&\leq |\I |+\wa{\bm{b}_1}+\w{v}\\
&\leq |\I |+\wa{\bm{b}_1}+\wa{\bm{b}_2}=|\I \cup \J|=d_2(\Co).
\end{align*}
Then, $d_2(\Co)=\wev{\mun, \bm{a}+\bm{b}_1-\bm{v}}=\wev{\mun,\bm{c}_g}$. Furthermore
\begin{align*}
\wa{\bm{c}_g}=\wa{\bm{a}+\bm{b}_1-\bm{v}}\leq \w{a}+\wa{\bm{b}_1}+\w{v}\leq \w{a}+\wa{\bm{b}_1}+\wa{\bm{b}_2}=\wa{\md}. 
\end{align*}
On the other hand, since $\bm{c}_g\in \bm{M}_2$ and $\md=\min_{\prec}\bm{M}_2$, we have  $\md\preceq \bm{c}_g$, so $\wa{\md}\leq \wa{\bm{c}_g}$. Hence, $\wa{\bm{c}_g}=\wa{\md}$.
\end{itemize}

In the case $q=2^s$, one always have $\bm{v}=\bm{b}_2$. Indeed, assume $\bm{v}\neq \bm{b}_2$. Then, $\xp{v}\prec \xp{b_2}$, i.e. $\xp{(-v)}\prec \xp{(-b_2)}$. We also have $\md\preceq \bm{c}_g$. Then 
\begin{align*}
\xp{\md}=\xpp{(a+b_1)}{(-b_2)}\preceq \xp{(\bm{a}+\bm{b}_1-\bm{v})}\preceq \xpp{(\bm{a}+\bm{b}_1)}{(-\bm{v})}.
\end{align*}
It follows that $\xp{(-b_2)}\preceq \xp{(-v)}$, a contradiction. Therefore, when $q=2^s$, we have $\bm{c}_g=\md$.

\medskip

\textbf{Case 2}. $\xp{b_2}\succ \xp{(a+b_1)}$.

Since $\wa{\bm{b}_2}\leq \wa{\bm{a}+\bm{b}_1}$, it follows that $\wa{\bm{b}_2}=\wa{\bm{a}+\bm{b}_1}$. Hence, $\wa{\md}=2\wa{\bm{b}_2}$.

We recall that $\suppa{\bm{b}_2}\cap\suppa{\bm{a}+\bm{b}_1}=\emptyset$, and we are using orders that satisfy the condition in \eqref{condicion-menos-orden}. Then,  $\xp{(-b_2)}\succ \xp{(-a-b_1)}$. 

Since $f=\xp{(-b_2)}-\xp{(-a-b_1)}\in \Ic$,  $\xp{(-b_2)}\in\inicC$. Consequently, there exists $g=\xp{(-u)}-\xp{(-v)}\in\GRx$ such that $\xp{(-v)}\prec\xp{(-u)}$ and $\xp{(-u)}\mid \xp{(-b_2)}$. Hence
\begin{equation}
\w{v}\leq \w{u}\leq \wa{\bm{b}_2}.\label{eq:aux-ineq1}
\end{equation}
It follows that
\begin{align}
\wa{\bm{c}_g}=\wa{\bm{v}-\bm{u}}\leq \w{v}+\w{u}\leq \wa{\bm{b}_2}+\wa{\bm{b}_2}=\wa{\md}. \label{eq:or-new-dq}
\end{align}

Consider now $\bm{c}=\md+\bm{u}-\bm{v}=\bm{a}+\bm{b}_1-\bm{b}_2+\bm{u}-\bm{v}=\bm{a}+\bm{b}_1-\bm{v}+(\bm{u}-\bm{b}_2)$. Then
\begin{align}
\xp{c}=\xp{(a+\bm{b}_1-v+u-\bm{b}_2)}&\preceq\xp{a}\xp{b_1}\xp{(-v)}\xp{(u-\bm{b}_2)}\nonumber\\
&\prec\xp{a}\xp{b_1}\xp{(-u)}\xp{(u-\bm{b}_2)}\nonumber\\
&=\xp{a}\xp{b_1}\xp{(-u)}\frac{\xp{(-b_2)}}{\xp{(-u)}} \qquad (\text{since $\xp{(-u)}\mid\xp{(-b_2)}$})\nonumber\\
&=\xp{a}\xp{b_1}\xp{(-b_2)}=\xp{(a+\bm{b}_1-\bm{b}_2)}=\xp{\md}.\label{eq:c2-1a}
\end{align}

On the other hand
\begin{align}
\wev{\mun,\bm{c}}&=|\suppa{\mun}\cup \suppa{\bm{c}}|\nonumber\\
&=|\suppa{\mun}\cup\suppa{\bm{a}+\bm{b}_1-\bm{v}+(\bm{u}-\bm{b}_2)}|\nonumber\\
&\leq |\I \cup \supp{a}\cup\suppa{\bm{b}_1}\cup\supp{v}\cup \suppa{\bm{u}-\bm{b}_2}|\nonumber\\
&=|\I \cup \suppa{\bm{b}_1}\cup\supp{v}\cup \suppa{\bm{u}-\bm{b}_2}| \quad (\text{since $\supp{a}=\I \cap \J$})\nonumber\\
&\leq |\I |+\wa{\bm{b}_1}+\w{v}+\wa{-\bm{b}_2-(-\bm{u})}\nonumber\\
&=  |\I |+\wa{\bm{b}_1}+\w{v}+\wa{-\bm{b}_2}-\wa{-\bm{u}}\quad (\text{since $\xp{(-u)}\mid \xp{(-b_2)}$})\nonumber\\
&=  |\I |+\wa{\bm{b}_1}+\w{v}+\wa{\bm{b}_2}-\w{u}\nonumber\\
&\leq |\I |+\wa{\bm{b}_1}+\wa{\bm{b}_2}\quad (\text{since $\w{v}\leq\w{u}$})\nonumber\\
&=|\I \cup \J|=d_2(\Co).\label{eq:c2-c1-c1}
\end{align}
Assume $\dim\ev{\mun,\bm{c}}=2$. By \eqref{eq:c2-c1-c1}, we have  $\wev{\mun,\bm{c}}=d_2(\Co)$. Then,  $\bm{c}\in \bm{M}_2$. Consequently, $\bm{m}_2=\min_{\prec} \bm{M}_2\preceq \bm{c}$. This contradicts \eqref{eq:c2-1a}.  Hence, $\bm{c}\in\ev{\mun}$.

Then, $\md+\bm{u}-\bm{v}=\lambda \mun$ for some $\lambda \in\Fn{q}$. It follows that $\bm{v}-\bm{u}=\md-\lambda\mun$. Hence
\begin{align}
\ev{\mun,\bm{c}_g}=\ev{\mun,\bm{v}-\bm{u}}=\ev{\mun,\md-\lambda\mun}=\ev{\mun,\md}. \label{eq:otra-20}
\end{align}
Then, $\wev{\mun,\bm{c}_g}=d_2(\Co)$. It remains to show that $\wa{\bm{c}_g}=\wa{\md}$. 

Since $\bm{c}_g\in \bm{M}_2$, $\md\preceq \bm{c}_g$. Then 
\begin{equation}
\wa{\md}\leq \wa{\bm{c}_g}. \label{eq:des-00a1}
\end{equation}
From  \eqref{eq:or-new-dq} and \eqref{eq:des-00a1} it follows that $\wa{\md}= \wa{\bm{c}_g}$.

Let us now prove that $\bm{c}_g=\md$ when $q=2^s$.

Notice that we always have $\bm{u}=\bm{b}_2$. Indeed, 
assume $\bm{u}\neq \bm{b}_2$. Then, $\w{u}\leq \wa{\bm{b}_2}-1$. Hence, $\wa{\bm{c}_g}=\wa{\bm{v}-\bm{u}} \leq \w{v}+\w{u} \leq\wa{\bm{b}_2}+ \wa{\bm{b}_2}-1<\w{\md}$. However, $\wa{\md}= \wa{\bm{c}_g}$. Therefore, $\bm{u}=\bm{b}_2$. Then, $g=\xp{(-b_2)}-\xp{(-v)}$. 

Let us prove that $\bm{v}=\bm{a}+\bm{b}_1$. Assume $\bm{v}\neq \bm{a}+\bm{b_1}$. Since $\md\preceq \bm{c}_g$, it follows that $
\xpp{(a+b_1)}{(-b_2)}=\xp{\md}\preceq \xp{(v-b_2)}\preceq \xpp{v}{(-b_2)}$. Hence, $\xp{(a+b_1)}\preceq \xp{v}$.

On the other hand, since  $\xp{(-b_2)}-\xp{(-a-b_1)},\xp{(-b_2)}-\xp{(-v)}\in \Ic$, it follows that $\xp{(-a-b_1)}-\xp{(-v)}\in \Ic$. Since $\xp{(-v)}\not\in\inicC$, we have $\xp{(-v)}\prec \xp{(-a-b_1)}$, i.e.,  $\xp{v}\prec \xp{(a+b_1)}$. This contradicts the fact that $\xp{(a+b_1)}\preceq \xp{v}$.  Then, $\bm{v}=\bm{a}+\bm{b}_1$.  Hence, $\bm{c}_g=\md$.
\end{proof}
\end{Proposition}

\textit{Note}. Regarding the conclusions in Proposition~\ref{prop-m2}, all the examples we have computed show that $\bm{c}_{g}=\md$ even when $q\neq 2^s$. However, to obtain the main result of this section (Theorem~\ref{corollary-main}), 
we only require the conditions $\wev{\mun,\bm{c}_g}=d_2(\Co)$ and $\wa{\bm{c}_g}=\wa{\md}$. 

\medskip

From Propositions~\ref{prop-m1},~\ref{prop-m2}, and Lemma~\ref{lemma-cg-m2}, we obtain the following theorem.

\begin{Teorema}\label{corollary-main}
Let $\Co\subset\F{q}{n}$ be a linear code such that $|\I \cap \J|\leq \tfrac{|\J|+1}{2}$. Let $\G$ be the reduced Gr\"obner basis of $\Ic$. Then, there exists $f,g\in \G$ such that $\dim\ev{\bm{c}_f,\bm{c}_g}=2$ and $d_2(\Co)=\wev{\bm{c}_f,\bm{c}_g}$,  where $\bm{c}_f,\bm{c}_g\in \crc{\Co}$. Hence, $M_{\G}$ is a $d_2$-test set.
\end{Teorema}

The linear code in the next example was taken from \cite[Example 4]{marquez-corbella_ideal_2016}.
\begin{Ejemplo}\label{ej1-d2-test-set}
Let $\Co\subset\F{3}{9}$ be a linear code with generator matrix
$$G=\begin{pmatrix}
1 & 0 & 0 & 0 & 0 & 1 & 0 & 2 & 0 \\
0 & 1 & 0 & 0 & 1 & 1 & 1 & 0 & 1 \\
0 & 0 & 1 & 1 & 2 & 2 & 1 & 1 & 0
\end{pmatrix}.
$$

Consider the order $\prec_{\text{rev}}$ where
$
\underbrace{x_{9,2} \prec x_{9,1}}_{X_9}\prec \underbrace{x_{8,2}\prec x_{8,1}}_{X_8}\prec\ldots\prec\underbrace{x_{1,2}\prec x_{1,1}}_{X_1}
$.

By using the computer algebra system \textsc{SageMath} \citep{sage_usage}, we obtain $$\mun=(2, 0, 0, 0, 0, 2, 0, 1, 0)\text{ and }\md=(0, 1, 0, 0, 1, 1, 1, 0, 1).$$ 

Then, $\I =\{1,6,8\}$ and $\J=\{2,5,6,7,9\}$. Hence, $d_2(\Co)=|\I   \cup \J|=7$. Therefore
$$|\I \cap \J|=1 <3=\frac{|\J|+1}{2}.$$

For the computation of the reduced Gr{\"o}bner basis $\G$ of $\Ic$, we implemented in \textsc{SageMath} \citep{sage_usage} the FGLM algorithm for monoid rings, as presented in \cite{borges2006generala}. Hence, $\G$ has $457$ binomials, of which $27$ are elements of $\Rx$. As proved in Theorem~\ref{corollary-main}, we have the binomials $f= x_{1,1}x_{8,2} - x_{6,2} \in\G$ and $g=x_{6,2}x_{7,2}x_{9,2} - x_{2,1}x_{5,1}\in \G$, where $\bm{c}_f=\mun$ and $\bm{c}_g=\md$. Then, $M_\G$ is a $d_2$-test set. 
\end{Ejemplo}

\section{On the existence of linear codes where $M_\G$ is not a \mbox{$d_2$-test} set}\label{s:mg-no-d2}

By the previous section, we know that there are linear codes $\Co\subset\F{q}{n}$ where $M_\G$ is a $d_2$-test set and $q>2$ (see Example~\ref{ej1-d2-test-set}). Unfortunately, this is not a general fact. In this section, we provide a family of linear codes  not satisfying this property. More precisely, we prove the following theorem:

\begin{Teorema}\label{main-result}
For each $q>2$, there exist linear codes $\Co\subset\F{q}{n}$ such that $M_{\G}$ is not a $d_2$-test set.
\end{Teorema}

The theorem is proved at the very end of this section. We need first several preliminary results and some constructions appearing in the following subsections.

\subsection{Construction of the linear code $\Co$}

In order to construct a linear code $\Co$ where $M_\G$ is not a $d_2$-test set, we begin by proving the following technical result, which will be essential for the subsequent construction of $\Co$.

\begin{Lemma}\label{l:desig}
Let $D'\subset \F{q}{m}$ be a linear code. Let $r=\left\lfloor \frac{|\J[D']|}{2} \right\rfloor$. Suppose $D'$ satisfies the following conditions:
\begin{align}
|\I[D']\cap \J[D']|&>\frac{|\J[D']|+1}{2}\text{ and } \label{eq:ine-im1}\\
|\I[D']|<|\J[D']|\quad  &\text{ or } \quad |\J[D']| \text{ is even}.
\label{eq:ine-im2}
\end{align}
Then
\begin{enumerate}
\item\label{d-1} $d_2(D')<3r$.
\item\label{d-2} $d_2(D')<|\I[D']|+r$.
\end{enumerate}
\end{Lemma}
\begin{proof}
We prove~\ref{d-1}. Suppose that $3 \pep{|\J[D']|}=3r \leq d_2(D')=|\I[D']|+|\J[D']|-|\I[D']\cap \J[D']|$.
Then
\begin{align}
|\I[D']\cap \J[D']|&\leq |\I[D']|+|\J[D']|- 3 \pe{|\J[D']|}\nonumber\\
&\leq |\I[D']|+|\J[D']|-3\frac{|\J[D']|-1}{2} \qquad \left(\text{since }\pep{|\J[D']|}\geq \tfrac{|\J[D']|-1}{2}\right),\nonumber\\
&\leq |\J[D']|+|\J[D']|-3\frac{|\J[D']|-1}{2}\qquad  (|\I[D']|\leq|\J[D']|),\nonumber\\
&=\frac{|\J[D']|+1}{2}+1. \label{eq:d-pr-des-1}
\end{align}
By \eqref{eq:ine-im1}, we must have $|\I[D']\cap \J[D']|=\tfrac{|\J[D']|+1}{2}+1$. Substituting this relation into \eqref{eq:d-pr-des-1} turns every inequality into an equality. In particular $|\I[D']|=|\J[D']|$ and $|\J[D']|$ is odd. However, this contradicts  \eqref{eq:ine-im2}. Thus, $d_2(D')<3r$.

Now we prove~\ref{d-2}. Since $|\I[D']\cap \J[D']|>\tfrac{|\J[D']|+1}{2}$, it follows that 
$\tfrac{|\J[D']|}{2}-\tfrac{1}{2}>|\J[D']|-|\I[D']\cap \J[D']|$. Moreover, given that $r=\pep{|\J[D']|}\geq  \tfrac{|\J[D']|}{2}-\tfrac{1}{2}$, we have 
$r>|\J[D']|-|\I[D']\cap \J[D']|$. Hence
\begin{equation*}
|\I[D']|+r>|\I[D']|+|\J[D']|-|\I[D']\cap \J[D']|=|\I[D']\cup \J[D']|=d_2(D'). \qedhere
\end{equation*}
\end{proof}

\begin{Remark}
Regarding the inequality (\ref{eq:ine-im1}), notice that in the binary case one always has $|\I[D']\cap \J[D']|\leq \frac{|\J[D']|+1}{2}$  (see \eqref{eq:desi-binary-case}).
\end{Remark}

In the following example, we show that, for any $q>2$, there are linear codes that satisfy \eqref{eq:ine-im1} and \eqref{eq:ine-im2}.
\begin{Ejemplo}
Let $D'=\ev{\bm{c}_1,\bm{c}_2}\subset\F{q}{2q+2}$, where $q>2$ and
\begin{align*}
\setcounter{MaxMatrixCols}{20}
\bm{c}_1&=\begin{pmatrix}
1&1&\alpha&\alpha^2&\ldots&\alpha^{q-1}&\alpha&\alpha^2&\ldots&\alpha^{q-1}&0&0
\end{pmatrix},\\
\bm{c}_2&=\begin{pmatrix}
0&0&1&1\phantom{\alpha}&\ldots&\phantom{\alpha}1\phantom{\alpha}&1&1\phantom{\alpha}&\ldots&\phantom{\alpha}1&\phantom{\alpha^2}1&1
\end{pmatrix}.
\end{align*}

For each $\bm{c}\in D'\setminus\{\bm{0}\} $, we have $\w{c}=2q$. This is clear if  $\bm{c}\in\ev{\bm{c}_1}$ or $\bm{c}\in\ev{\bm{c}_2}$. Now, if $\bm{c}\in D'\setminus (\ev{\bm{c}_1}\cup \ev{\bm{c}_2})$, then $\bm{c}=\alpha^{i_1}\bm{c}_1+\alpha^{i_2}\bm{c}_2$ for some $\alpha^{i_i},\alpha^{i_2}\in \Fn{q}$. Let $i=i_2-i_1$, then 
$$\w{c}=\wa{\alpha^{i_1}\bm{c}_1+\alpha^{i_2}\bm{c}_2}=\wa{\alpha^{i_1}(\bm{c}_1+\alpha^i\bm{c}_2)}=\wa{\bm{c}_1+\alpha^i\bm{c}_2},$$
where
\begin{align*}
\setcounter{MaxMatrixCols}{20}
\bm{c}_1+\alpha^i\bm{c}_2=\begin{pmatrix}
1&1&\alpha+\alpha^i&\ldots&\alpha^{q-1}+\alpha^i&\alpha+\alpha^i&\ldots&\alpha^{q-1}+\alpha^i&\alpha^i&\alpha^i
\end{pmatrix}.
\end{align*}

Since there exists exactly one $\alpha^j$ such that $\alpha^j+\alpha^i=0$, we have $$\w{c}=\wa{\bm{c}_1+\alpha^i\bm{c}_2}=(2q+2)-2=2q.$$
Hence, $\w{c}=2q$ for each $\bm{c}\in D'\setminus\{\bm{0}\}$.

Thus, for any degree compatible order we have $|\I[D']|=|\J[D']|=2q$. Since $|\I[D']\cup \J[D']|=d_2(D')=\wa{D'}=2q+2$, we have $|\I[D']\cap \J[D']|=|\I[D]'|+|\J[D']|-|\I[D']\cup\J[D']|=2q-2$. Also,  $|\I[D']|$, $|\J[D']|$ are even and
$$|\I[D']\cap \J[D']|=2q-2>\frac{2q+1}{2}=\frac{|\J[D']|+1}{2} \qquad (q> 2).$$
Hence, $D'$ satisfies \eqref{eq:ine-im1} and \eqref{eq:ine-im2}.
\end{Ejemplo}

In what follows we construct a special linear code $\Co\subset\F{q}{n}$, for each $q>2$. This code will be used  to prove Theorem~\ref{main-result} in the next subsection. The construction of such code is divided in four steps.

Moreover, since we will be working with various linear codes, for each linear code $\mathcal{D}$ we write $\mun(\mathcal{D})$ and $\md(\mathcal{D})$ instead of $\mun$ and $\md$. Likewise, we denote by $\bm{M}_1(\mathcal{D})$ and $\bm{M}_2(\mathcal{D})$ the sets defined in Section~\ref{s:d2-test-set}.

\medskip

\noindent \step\label{step-1} We first construct an auxiliary code $D'\subset \F{q}{m}$ with some prescribed conditions.

Let $D'=\ev{\bm{c}_1',\bm{c}_2'}\subset \F{q}{m}$ where $\dim (D')=2$. Let $K$ be a field and let  $\prec_{m}$ be an order on the monomials of $R_m=K[x_{i,j}:1\leq i\leq m, 1\leq j\leq q-1]$. The order in the indeterminates is given by: $X_m\prec_m \ldots \prec_m X_2\prec_m X_1$. Moreover, using the order $\prec_{m}$ we have 
$$\bm{m}_1(D')=\bm{c}_1' \text{ and } \bm{m}_2(D')=\bm{c}_2'.$$
Hence, by Remark~\ref{re:M1-M2-dim2}, we obtain
\begin{equation}
\bm{c}_1'=\min_{\prec_m}(D'\setminus\{\bm{0}\}) \text{ and } \bm{c}_2'=\min_{\prec_m}(D'\setminus\ev{\bm{c}'_1}).\label{eq:m1-m2-D-prime}
\end{equation}
It follows that  $d_2(D')=\wa{D'}=\wev{\bm{c}_1',\bm{c}_2'}=|\suppa{\bm{c}_1'}\cup \suppa{\bm{c}_2'}|=|\I[D']\cup \J[D']|$.

Let $r= \pep{\w{c_2'}} =\pep{|\J[D']|}$. Then
\begin{equation}
\w{\bm{c}_2'}\geq 2r.
\label{eq:des-r1}
\end{equation}
Moreover, assume that $D'$ satisfies the hypothesis on Lemma~\ref{l:desig}. Then 
\begin{align}
d_2(D')&<3r,  \label{eq:des-en-r}\\
d_2(D')&<|\I[D']|+r. \label{eq:des-con-I}
\end{align}

\noindent\step\label{step-2}  Building on the notation introduced in Step~\ref{step-1}, we now define a subset $P\subset\F{q}{m}$, together with collections of vectors $\bm{c}_i$, $\bm{u}_j$ and $\bm{v}_j$, whose index ranges will be specified below. These objects will be used in the construction of $\Co$.

Consider the set 
\begin{equation}
P=\left\{\bm{u}'\in \F{q}{m}\mid \suppa{\bm{u}'}\subset \suppa{D'} \text{ and } \wa{\bm{u}'}=r\right\}.\label{definition-P}
\end{equation} 
Let $\ell=|P|={\wa{D'} \choose r}\cdot (q-1)^{r}$. 

We now order the elements of $P$ using the order $\prec_m$:
$$\bm{u}_{\ell}'\prec_m\bm{u}_{\ell-1}'\prec_m\ldots \prec_m\bm{u}_{2}' \prec_m\bm{u}_{1}'.$$
We now define the following vectors. For each $i\in[\ell]$
\begin{align}
 \bm{v}_i'&=(\underbrace{0\ \ldots \ 0}_{m}\ \underbrace{0\ \ldots \ 0}_{(i-1)r}\  \underbrace{1\ldots 1}_{r})\in \F{q}{m+ir}. \label{def-vi-prima}
\end{align}
Let $n=m+\ell r$. Consider the function 
\begin{align*}
\iota_{m,n}:\F{q}{m}&\to \F{q}{n}\\
\bm{c}&\mapsto (\bm{c} \  \underbrace{0\ldots 0}_{n-m}).
\end{align*} 

\begin{Remark}\label{remark-funcion-inclusion}
For a word $\bm{c}\in\F{q}{m}$ we have $\suppa{\iota_{m,n}(\bm{c})}=\supp{c}\subset[m]$. Also, for a vector subspace $V\subset\F{q}{m}$, we have $\suppa{\iota_{m,n}(V)}=\suppa{V}$.
\end{Remark}
For each $i\in\{1,2\}$ and $j\in [\ell]$, we define the vectors
\begin{align*}
\bm{c}_i=\iota_{m,n}(\bm{c}_i'),  \qquad
\bm{u}_j=\iota_{m,n}(\bm{u}_j'), \qquad 
\bm{v}_j=\iota_{m+jr,n}( \bm{v}_j'). 
\end{align*}

Notice that from \eqref{eq:des-r1} and Remark~\ref{remark-funcion-inclusion} we have
\begin{equation}
\wa{\bm{c}_2}=\wa{\bm{c}_2'}\geq 2r. \label{eq:des-c2-r}
\end{equation}

From the definition of $\bm{v}_i'$ in \eqref{def-vi-prima}, it follows that $\wa{\bm{v}_i}=r$. Moreover
\begin{equation}
\suppa{\bm{v}_i}\cap \suppa{\bm{v}_j}=\emptyset, \qquad i\neq j. \label{sop-disjunto-vi}
\end{equation}

\noindent \step\label{step-3}With the notation from Steps~\ref{step-1} and~\ref{step-2}, we fix an order on the monomials of $\Kx=K[x_{i,j}: 1\leq i\leq n,\ 1\leq j\leq q-1]$, satisfying additional conditions.

Fix a degree compatible order $\prec$ on the monomials of 
$\Kx$. The order of the indeterminates is given by \eqref{orden-indeterminadas}. Moreover, we require this order to be compatible with the monomial order $\prec_m$ used over the monomials of $R_m$. In other words, let $r_1$ and $r_2$ be monomials in $R_m$. Then, seen as monomials in $\Kx$, we have
\begin{equation}
r_1\prec r_2\Leftrightarrow r_1\prec_m r_2.
\label{eq:rela-order-m}
\end{equation}

Additionally, we require that this order satisfies the condition: for $\bm{u},\bm{v}\in\F{q}{n}$, if $\w{u}=\w{v}$, $\supp{u}\subset\{1,2,\ldots,m\}$ and $\supp{v}\cap\{1,2,\ldots,m\}=\emptyset$, then $\xp{u}\succ \xp{v}$. The orders $\prec_{\text{lex}}$ and $\prec_{\text{rev}}$ satisfy the required condition, since we are using the order in the indeterminates given by: $X_n\prec \ldots \prec X_m\prec \ldots \prec X_2\prec X_1$. 

\medskip

\noindent \step\label{step-4} We now proceed to define $\Co$ using the notation introduced in Steps~\ref{step-1}-\ref{step-3}.

Define a linear code $\Co$ as follows:
\begin{equation}
\Co=\langle \bm{c}_1,\bm{c}_2,\bm{u}_1-\bm{v}_1,\bm{u}_2-\bm{v}_2,\ldots,\bm{u}_{\ell}-\bm{v}_{\ell}   \rangle\subset\F{q}{n}.
\label{eq:codigo-general}
\end{equation}

Since $\wa{\bm{u}_i}=\wa{\bm{v}_i}$, $\suppa{\bm{u}_i}\subset\{1,2,\ldots,m\}$ and $\suppa{\bm{v}_i}\cap\{1,2,\ldots,m\}=\emptyset$, we have $\xp{u_i}\succ \xp{v_i}$. Also, since $\xp{u_i}-\xp{v_i}\in\Ic$ for each $i\in[\ell]$, we have
\begin{equation}
\xp{u_i}\in \inicC,\qquad i\in[\ell].\label{inic-ui}  
\end{equation}

Let $U=\{\bm{u}_1-\bm{v}_1,\bm{u}_2-\bm{v}_2,\ldots,\bm{u}_{\ell}-\bm{v}_{\ell}\}$. We now show that the set $\{\bm{c}_1,\bm{c}_2\}\cup U$ is linearly independent.  Let $\lambda_1,\lambda_2,\theta_1,\ldots,\theta_\ell \in \Fn{q}$ be such that
\begin{align*}
\lambda_1\bm{c}_1+\lambda_2\bm{c}_2+\sum_{i=1}^{\ell}\theta_i(\bm{u}_i-\bm{v}_i)&=(\underbrace{0\ldots0}_{m}\ \underbrace{0\ldots0}_{r}\ \ldots \ \underbrace{0\ldots0}_{r})\\
(\underbrace{\lambda_1\bm{c}_1'+\lambda_2\bm{c}_2'+\sum_{i=1}^{\ell}\theta_i\bm{u}_i'}_{m}\ \underbrace{(-\theta_1)\ldots  (-\theta_1)}_{r} \ldots \ \underbrace{(-\theta_{\ell})\ldots  (-\theta_\ell)}_{r}) &=(\underbrace{0\ldots0}_{m}\ \underbrace{0\ldots0}_{r}\ \ldots \ \underbrace{0\ldots0}_{r}).
\end{align*}
Then, $\theta_1=\ldots =\theta_{\ell}=0$. This implies that in the first $m$ components we have $\lambda_1\bm{c}_1'+\lambda_2\bm{c}_2'=\bm{0}\in \F{q}{m}$. Since $\{\bm{c}_1',\bm{c}_2'\}$ is linearly independent, it follows that $\lambda_1=\lambda_2=0$.

\subsection{Second generalized Hamming weight of $\Co$}

In this subsection we continue the discussion of the code $\Co$ introduced above. In particular, we show that there is a unique two-dimensional subspace of $\Co$ whose weight coincides with $d_2(\Co)$. 

\begin{Proposition}\label{d2C}
With the notation of the previous subsection, let $D=\ev{\bm{c}_1,\bm{c}_2}\subset\Co\subset\F{q}{n}$. Then, $\wa{D}=d_2(\Co)$ and $D$ is the only two-dimensional subspace of $\Co$ with this property.
\begin{proof}
To prove the Proposition, we first compute $\wa{D}$. By Remark~\ref{remark-funcion-inclusion} we have
\begin{align}
\wa{D}=\wev{\bm{c}_1,\bm{c}_2}&=|\suppa{\bm{c}_1}\cup \suppa{\bm{c}_2}|\nonumber\\
&=|\suppa{\bm{c}_1'}\cup\suppa{\bm{c}_2' }|=\wa{D'}=d_2(D'). \label{eq:ine2-c1-c2} 
\end{align}
Also, for any $\bm{c}\in D\setminus\{\bm{0}\}$, we have $\bm{c}=\mu_1\bm{c}_1+\mu_2\bm{c}_2$ for some $\mu_1,\mu_
2\in\Fn{q}$. Then
\begin{align}
\w{c}&=\wa{\mu_1\bm{c}_1+\mu_2\bm{c}_2}=\wa{\mu_1\bm{c}_1'+\mu_2\bm{c}_2'}\geq \wa{\bm{c}_1'}=|\I[D']|. \label{eq:des-2r-1}
\end{align}

We now find the second generalized Hamming weight of $\Co$. We do this by showing that the weight of every two-dimensional subspace of $\Co$ is greater or equal than the weight of $D$.

Let $\bm{x}_1,\bm{x}_2\in\Co$, where $\dim(\ev{\bm{x}_1,\bm{x}_2})=2$. Then, for $i=1,2$ we have
\begin{align*}
\bm{x}_i&=\lambda_1^{(i)}\bm{c}_1+\lambda_2^{(i)}\bm{c}_2
+\sum_{j=1}^{N_i}\theta_{i_j}(\bm{u}_{i_j}-\bm{v}_{i_j}), \end{align*}
for some $\bm{u}_{i_j}-\bm{v}_{i_j}\in U $, $\lambda_1^{(i)},\lambda_2^{(i)}\in \Fn{q}$, $\theta_{i_j}\in \Fn{q}\setminus\{0\}$ and $N_i\in \mathbb{N}\cup\{0\}$. Now, for $i=1,2$ we define
\begin{align*}
\bm{x}_i'&=\lambda_1^{(i)}\bm{c}_1+\lambda_2^{(i)}\bm{c}_2
+\sum_{j=1}^{N_i}\theta_{i_j}\bm{u}_{i_j},\\
A_i&=\{\bm{u}_{i_j}-\bm{v}_{i_j}:1\leq j\leq N_i\},\\
B_i&=\suppa{-\theta_{i_j}\sum_{j=1}^{N_i}\bm{v}_{i_j}}=\bigcup_{j=1}^{N_i}\suppa{\bm{v}_{i_j}}, \qquad (\text{by }\eqref{sop-disjunto-vi}). 
\end{align*}

Since $\suppa{\bm{c}_1}$, $ \suppa{\bm{c}_2}$, $\suppa{\bm{u}_{i_j}}\subset [m]$ and $\suppa{\bm{v}_{i_j}}\cap [m]=\emptyset$, it follows that  $\suppa{\bm{x}_i}=\suppa{\bm{x}_i'}\sqcup B_i$. In addition, by \eqref{sop-disjunto-vi} we have $|B_i|=rN_i$. Also, if $A_1\cap A_2=\emptyset$, then $B_1\cap B_2=\emptyset$.

Now we have
\begin{align*}
\wev{\bm{x}_1,\bm{x}_2}&=|\suppa{\bm{x}_1}\cup \suppa{\bm{x}_2}|\\
&=|\suppa{\bm{x}_1'}\cup B_1\cup \suppa{\bm{x}_2'}\cup B_2|\\
&=|\suppa{\bm{x}_1'}\cup \suppa{\bm{x}_2'}|+|B_1\cup B_2|.
\end{align*}
Hence, $\wev{\bm{x}_1,\bm{x}_2}\geq |B_i|=rN_i$ for $i=1,2$. If $N_1\geq 3$ or $N_2\geq 3$, then $\wev{\bm{x}_1,\bm{x}_2}\geq 3r$. In these cases, by \eqref{eq:des-en-r} and  \eqref{eq:ine2-c1-c2} it follows that $\wev{\bm{x}_1,\bm{x}_2}\geq 3r>d_2(D')=\wa{D}$.
Hence, to find $d_2(\Co)$ it suffices to study the cases where $N_1,N_2\in \{0,1,2\}$.
\medskip

{
\setlength{\parindent}{0pt}
\case\label{caso:1} Let $N_1=N_2=0$. 

In this case, $\bm{x}_1,\bm{x}_2\in \ev{\bm{c}_1,\bm{c}_2}$ and so $\ev{\bm{x}_1,\bm{x}_2}=D$. Hence, $\wev{\bm{x}_1,\bm{x}_2}=\wa{D}$.
}
\medskip

Let us emphasize that, in all the cases that follow, the corresponding weight is strictly greater than $\wa{D}$.

\medskip

{
\setlength{\parindent}{0pt}
\case\label{case:2} Either $N_1=0$ or $N_2=0$, but not both.

Without loss of generality, we assume $N_2=0$. Hence,  $\bm{x}_2\in\ev{\bm{c}_1,\bm{c}_2}\setminus\{\bm{0}\}$. By \eqref{eq:des-2r-1} we have $\wa{\bm{x}_2}\geq |\I[D']|$. Also, since $N_1\in \{1,2\}$, we have $|B_1|=rN_1\geq r$.  Hence 
\begin{align*}
\wev{\bm{x}_1,\bm{x}_2}&=|\suppa{\bm{x}_ 1'}\cup \suppa{\bm{x}_ 2}|+ |B_1|\\
&\geq \wa{\bm{x}_2}+|B_1|
\geq \wa{\bm{x}_2}+r
\geq |\I[D']|+r
>d_2(D')=\wa{D} \quad (\text{by \eqref{eq:des-con-I}}). 
\end{align*}

\case\label{caso:3} Let $N_1=N_2=1$.
\begin{align*}
\bm{x}_i=\lambda_1^{(i)}\bm{c}_1+\lambda_2^{(i)}\bm{c}_2
+\theta_{i_1}(\bm{u}_{i_1}-\bm{v}_{i_1}), \qquad i=1,2.
\end{align*}

\subcase\label{subcase:3.1} $A_1=\{\bm{u}_{1_1}-\bm{v}_{1_1}\}=\{\bm{u}_{2_1}-\bm{v}_{2_1}\}=A_2$.

Let $\mu\in\Fn{q}$ be such that $\theta_{1_1}+\mu\theta_{2_1}=0$. 

Since $\dim(\ev{\bm{x}_1,\bm{x}_2})=2$, we have $\bm{x}_1+\mu\bm{x}_2=(\lambda_1^{(1)}+\mu\lambda_1^
{(2)})\bm{c}_1+(\lambda_2^{(1)}+\mu\lambda_2^{(2)})\bm{c}_2\neq\bm{0}$.

Also, $\wev{\bm{x}_1,\bm{x}_2}=\wev{\bm{x}_1,\bm{x}_1+\mu\bm{x}_2}$, so we are in Case~\ref{case:2}. Then, $\wev{\bm{x}_1,\bm{x}_2}>\wa{D}$.

\subcase\label{subcase:3.2} $A_1\neq A_2$. Then, $B_1\cap B_2=\emptyset$.
\begin{align*}
\wev{\bm{x}_1,\bm{x}_2}&= |\suppa{\bm{x}_{1}'}\cup \suppa{\bm{x}_{2}'}|+|B_1\cup B_2|\\ 
&= |\suppa{\bm{x}_{1}'}\cup \suppa{\bm{x}_{2}'}|+| B_1|+| B_2|\\
&= |\suppa{\lambda_1^{(1)}\bm{c}_1+\lambda_2^{(1)}\bm{c}_2
+\theta_{1_1}\bm{u}_{1_1}}\cup \suppa{\bm{x}_{2}'}|+2r\\
&\geq \wa{\lambda_1^{(1)}\bm{c}_1+\lambda_2^{(1)}\bm{c}_2
+\theta_{1_1}\bm{u}_{1_1}}+2r.
\end{align*}

\subsubcase\label{caso:3.2.1} $\lambda_1^{(1)}\bm{c}_1+\lambda_2^{(1)}\bm{c}_2=\bm{0}$.
\begin{align*}
\wev{\bm{x}_1,\bm{x}_2}\geq  \wa{\theta_{1_1}\bm{u}_{1_1}}+2r= r+2r=3r>d_2(D')=\wa{D}  \qquad (\text{by \eqref{eq:des-en-r}}).
\end{align*}

\subsubcase\label{caso:3.2.2} $\lambda_1^{(1)}\bm{c}_1+\lambda_2^{(1)}\bm{c}_2\neq \bm{0}$.
\begin{align*}
\wev{\bm{x}_1,\bm{x}_2}&\geq \wa{\lambda_1^{(1)}\bm{c}_1+\lambda_2^{(1)}\bm{c}_2
+\theta_{1_1}\bm{u}_{1_1}}+2r\\
&\geq\wa{\lambda_1^{(1)}\bm{c}_1+\lambda_2^{(1)}\bm{c}_2}-\wa{\theta_{1_1}\bm{u}_{1_1}}+2r\\
&=\wa{\lambda_1^{(1)}\bm{c}_1+\lambda_2^{(1)}\bm{c}_2}+r\qquad (\text{since }\wa{\theta_{1_1}\bm{u}_{1_1}}=r),\\
&\geq |\I[D']|+r \qquad (\text{by \eqref{eq:des-2r-1}}),\\
&>\ d_2(D') \qquad (\text{by \eqref{eq:des-con-I}}),\\
&=\wa{D}.
\end{align*}

\case\label{case:4} Let  $N_1,N_2\in \{1,2\}$,  where $N_1\neq N_2$.

Without loss of generality, we assume $N_1=2$ and $N_2=1$.
\begin{align*}
\bm{x}_1&=\lambda_1^{(1)}\bm{c}_1+\lambda_2^{(1)}\bm{c}_2+\theta_{1_1}(\bm{u}_{1_1}-\bm{v}_{1_1})+\theta_{1_2}(\bm{u}_{1_2}-\bm{v}_{1_2}),\\
\bm{x}_2&=\lambda_1^{(2)}\bm{c}_1+\lambda_2^{(2)}\bm{c}_2+\theta_{2_1}(\bm{u}_{2_1}-\bm{v}_{2_1}).
\end{align*}

\subcase\label{case:4.1} $A_1\cap A_2\neq \emptyset$.  Without loss of generality, let $\bm{u}_{1_1}-\bm{v}_{1_1}=\bm{u}_{2_1}-\bm{v}_{2_1}$. Then,   
\begin{align*}
\bm{x}_1&=\lambda_1^{(1)}\bm{c}_1+\lambda_2^{(1)}\bm{c}_2+\theta_{1_1}(\bm{u}_{1_1}-\bm{v}_{1_1})+\theta_{1_2}(\bm{u}_{1_2}-\bm{v}_{1_2}),\\
\bm{x}_2&=\lambda_1^{(2)}\bm{c}_1+\lambda_2^{(2)}\bm{c}_2+\theta_{2_1}(\bm{u}_{1_1}-\bm{v}_{1_1}).
\end{align*}
Let $\mu\in\Fn{q}$ be such that $\theta_{1_1}+\mu\theta_{2_1}=0$. 

Since $\dim(\ev{\bm{x}_1,\bm{x}_2})=2$, it follows that $$\bm{x}_1+\mu\bm{x}_2=(\lambda_1^{(1)}+\mu\lambda_1^{(2)})\bm{c}_1+(\lambda_2^{(1)}+\mu\lambda_2^{(2)})\bm{c}_2+\theta_{1_2}(\bm{u}_{1_2}-\bm{v}_{1_2})\neq\bm{0}.$$
Since $\wev{\bm{x}_1,\bm{x}_2}=\wev{\bm{x}_2,\bm{x}_1+\mu\bm{x}_2}$, we are back in Case~\ref{subcase:3.2}. Then, $\wev{\bm{x}_1,\bm{x}_2}>\wa{D}$.\\

\subcase\label{case:4.2} $A_1\cap A_2= \emptyset$. Then, $B_1\cap B_2=\emptyset$.
\begin{align*}
\wev{\bm{x}_1,\bm{x}_2}&= |\suppa{\bm{x}_{1}'}\cup \suppa{\bm{x}_{2}'}\cup B_1\cup B_2|\\ 
&= |\suppa{\bm{x}_{1}'}\cup \suppa{\bm{x}_{2}'}|+| B_1|+| B_2|\\
&= |\suppa{\bm{x}_{1}'}\cup \suppa{\bm{x}_{2}'}|+2r+r\\
&\geq  3r>d_2(D')=\wa{D} \qquad (\text{by \eqref{eq:des-en-r}}).
\end{align*}

\case\label{case:5} Let $N_1=N_2=2$.
\begin{align*}
\bm{x}_1&=\lambda_1^{(1)}\bm{c}_1+\lambda_2^{(1)}\bm{c}_2+\theta_{1_1}(\bm{u}_{1_1}-\bm{v}_{1_1})+\theta_{1_2}(\bm{u}_{1_2}-\bm{v}_{1_2}),\\
\bm{x}_2&=\lambda_1^{(2)}\bm{c}_1+\lambda_2^{(2)}\bm{c}_2+\theta_{2_1}(\bm{u}_{2_1}-\bm{v}_{2_1})+\theta_{2_2}(\bm{u}_{2_2}-\bm{v}_{2_2}).
\end{align*}

In general we have
\begin{align*}
\wev{\bm{x}_1,\bm{x}_2}=|\suppa{\bm{x}_1'}\cup \suppa{\bm{x}_2'}|+|B_1|+|B_2|-|B_1\cap B_2|.
\end{align*}
In this case $|B_1|=|B_2|=2r$. If $|B_1\cap B_2|\leq r$, then $|B_1|+|B_2|-|B_1\cap B_2|\geq 3r>d_2(D')$. Hence, we have $\wev{\bm{x}_1,\bm{x}_2}>d_2(D')=\wa{D}$.

So, the only remaining case is $|B_1\cap B_2|=2r$. This happens only when $A_1=A_2$. In this case we have
\begin{align*}
\bm{x}_1&=\lambda_1^{(1)}\bm{c}_1+\lambda_2^{(1)}\bm{c}_2+\theta_{1_1}(\bm{u}_{1_1}-\bm{v}_{1_1})+\theta_{1_2}(\bm{u}_{1_2}-\bm{v}_{1_2}),\\
\bm{x}_2&=\lambda_1^{(2)}\bm{c}_1+\lambda_2^{(2)}\bm{c}_2+\theta_{2_1}(\bm{u}_{1_1}-\bm{v}_{1_1})+\theta_{2_2}(\bm{u}_{1_2}-\bm{v}_{1_2}).
\end{align*}

Let $\mu\in\Fn{q}$ be such that $\theta_{1_1}+\mu\theta_{2_1}=0$. Then $$\bm{x}_1+\mu\bm{x}_2=(\lambda_1^{(1)}+\mu\lambda_1^{(2)})\bm{c}_1+(\lambda_2^{(1)}+\mu\lambda_2^{(2)})\bm{c}_2+(\theta_{1_2}+\mu\theta_{2_2})(\bm{u}_{1_2}-\bm{v}_{1_2})\neq\bm{0}.$$

We have $\wev{\bm{x}_1,\bm{x}_2}=\wev{\bm{x}_1,\bm{x}_1+\mu\bm{x}_2}$.

\subcase\label{case:5.1} $\theta_{1_2}+\mu\theta_{2_2}=0$.

In this case, we are back in Case~\ref{case:2}. Then,  $\wev{\bm{x}_1,\bm{x}_2}>\wa{D}$.

\subcase\label{case:5.2} $\theta_{1_2}+\mu\theta_{2_2}\neq 0$.

Now we are in Case~\ref{case:4.1}, so again $\wev{\bm{x}_1,\bm{x}_2}>\wa{D}$.
}

\medskip

As a result of the outcome of the previous cases we conclude $d_2(\Co)=\wa{D}$. Furthermore, the previous analysis also showed that $D=\ev{\bm{c}_1,\bm{c}_2}$ is the only two-dimensional subspace of $\Co$ such that $\wa{D}=d_2(\Co)$. 
\end{proof}
\end{Proposition}

\begin{Corollary}\label{m1-m2-codigo-general}
$\mun(\Co)=\bm{c}_1$ and $\md(\Co)=\bm{c}_2$.
\begin{proof}
Since $\dim(D)=\dim(\ev{\bm{c}_1,\bm{c}_2})=2$, by Remark~\ref{re:M1-M2-dim2} we have
$$\bm{m}_1(D)=\min_{\prec}(D\setminus\{\bm{0}\}) \text{ and } \bm{m}_2(D)=\min_{\prec}(D\setminus\ev{\bm{m}_1(D)}).$$

Then, in view of \eqref{eq:m1-m2-D-prime}, \eqref{eq:rela-order-m}, and the definition of  $\bm{c}_1$, it follows that $\bm{m}_1(D)=\bm{c}_1$.
Hence, $\bm{m}_2(D)=\min_{\prec}(\ev{\bm{c}_1,\bm{c}_2}\setminus\ev{\bm{c}_1})$.
Similarly, we conclude that $\bm{m}_2(D)=\bm{c}_2$.

Now we prove that $\bm{c}_1=\min_{\prec}(\bm{M}_1(\Co))$. Let $\bm{m}\in \bm{M}_1(\Co)$. Then, $\bm{m}\neq \bm{0}$ and there exists $\bm{m}'\in\Co$ such that $\wev{\bm{m},\bm{m}'}=d_2(\Co)$. By Proposition~\ref{d2C}, $D$ is the only two-dimensional subspace of $\Co$ satisfying that its weight is $d_2(\Co)$.  Hence, if $\dim\ev{\bm{m},\bm{m}'}=2$, then $\ev{\bm{m},\bm{m}'}=D$. Since $\bm{c}_1=\min_{\prec }(D\setminus\{\bm{0}\})$ and $\bm{m}\neq \bm{0}$, it follows that $\bm{c}_1\preceq \bm{m}$. On the other hand, if $\bm{m}'\in\ev{\bm{m}}$, then $d_2(\Co)=\wev{\bm{m},\bm{m}'}=\w{m}$. By Lemma~\ref{lemma-m1-d2} we have $\wa{\bm{c}_1}=\wa{\mun(D)}<d_2(D)=d_2(\Co)=\w{m}$. Then, $\wa{\bm{c}_1}<\w{m}$. Hence, $\bm{c}_1\prec \bm{m}$.

From the previous discussion and the fact that $\bm{c}_1\in \bm{M}_1(\Co)$ we conclude that $\bm{c}_1=\min_{\prec}(\bm{M}_1(\Co))$. Then, $\bm{m}_1(\Co)=\bm{c}_1$.

Therefore, $\bm{M}_2(\Co)=\left\{\bm{m} \in\Co \mid d_{\ii}(\mathcal{C})=\wev{\bm{c}_1, \bm{m}}\right\}$. Let $\bm{m}\in \bm{M}_2(\Co)$.
Since $\wa{\bm{c}_1}<d_2(\Co)$, $\dim(\ev{\bm{c}_1,\bm{m}})=2$. Moreover, since $d_{\ii}(\mathcal{C})=\wev{\bm{c}_1, \bm{m}}$, it follows that $\ev{\bm{c}_1, \bm{m}}=D$. Hence,  $\bm{M}_2(\Co)=D\setminus\ev{\bm{c}_1}$. Then, $\bm{m}_2(\Co)=\bm{m}_2(D)=\bm{c}_2$.
\end{proof}
\end{Corollary}

\begin{Proposition}\label{out-c1}
With the notation of the previous subsection, let $\G$ be the reduced Gr\"obner basis of $\Ic$. For $\bm{c}\in \ev{\bm{c}_1,\bm{c}_2} \setminus \ev{\bm{c}_1}$, there is no $f\in \G$ such that $\bm{c}_f=\bm{c}$.
\begin{proof}
Let $\bm{c}\in \ev{\bm{c}_1,\bm{c}_2} \setminus \ev{\bm{c}_1}$. Then,  $\supp{c}\subset \suppa{\bm{c}_1}\cup\suppa{\bm{c}_2}=\suppa{D}=\suppa{D'}$. Also, since
$\ev{\bm{c}_1,\bm{c}_2}=\ev{\bm{c}_1,\bm{c}}$, by Proposition~\ref{d2C},  $\wev{\bm{c}_1,\bm{c}}=d_2(\Co)$.  
Thus, $\bm{c}\in \bm{M}_2(\Co)$. By Corollary~\ref{m1-m2-codigo-general} we have $\bm{m}_2(\Co)=\bm{c}_2$.  Then, $\bm{c}_2=\min_{\prec}(\bm{M}_2(\Co))\leq \bm{c}$. Hence, by \eqref{eq:des-c2-r} we have
\begin{equation}
2r\leq \wa{\bm{c}_2}\leq\w{c}.
\label{eq:m2-peso-r}
\end{equation} 

For the sake of contradiction, we assume that there exists $f\in\G$ such that $\bm{c}_f=\bm{c}$. Since $\bm{c}\neq \bm{0}$, it follows that $f\in\GRx$. Then, by Proposition~\ref{prop-division-binomios-grobner}, we know that $f$ has the form $f=\xp{a}-\xp{b}$, where $\xp{b}\prec\xp{a}$ and $\w{b}\leq \w{a}\leq \w{b}+1$. Since $f\in\GRx$, then
\begin{equation}
\xp{b}\notin\inicC. \label{eq:b-not-in-inicC}
\end{equation}

Since $\bm{c}_f=\bm{c}=\bm{a}-\bm{b}$, by Lemma~\ref{soportes-rem} we have $\supp{c}=\supp{a}\cup \supp{b}$. Hence, $\supp{b}\subset \supp{c}\subset\suppa{D'}$.

If $\w{b}<r$, then $\w{b}\leq r-1$. Hence, $\w{a}\leq r$. It follows that $\w{c}=\w{a-b}\leq \w{a}+\w{b}\leq r+r-1=2r-1$. But, from \eqref{eq:m2-peso-r} we have $\w{c}\geq 2r$. Thus, $\w{b}\geq r$.

We recall that the set $P$ defined in \eqref{definition-P} contains all the words $\bm{u}_i'\in\F{q}{m}$  such that $\suppa{\bm{u}_i'}\subset\suppa{D'}$ and $\w{\bm{u}_i'}=r$. Then, we can find $\bm{u}_j'$ such that $\xp{u_j'}\mid \xp{b}$.

Moreover, $\xp{u_j'}=\xp{u_j}$ where $\bm{u}_j-\bm{v}_j \in U$. Hence, $\xp{u_j}\mid \xp{b}$. From \eqref{inic-ui}, we have $\xp{u_j}\in \inicC$. Then, $\xp{b}\in \inicC$, but this contradicts \eqref{eq:b-not-in-inicC}.
\end{proof}
\end{Proposition}

We are now in position of proving Theorem~\ref{main-result}.

\begin{proof}(of Theorem~\ref{main-result})
Let $\Co$ be the linear code defined in \eqref{eq:codigo-general}.  

Suppose $M_\G$ is a $d_2$-test set. Then, there exist  $f,g\in \G$ such that $\dim\ev{\bm{c}_f,\bm{c}_g}=2$ and $d_2(\Co)=\wev{\bm{c}_f,\bm{c}_g}$. By Proposition~\ref{d2C}, we have $\ev{\bm{c}_f,\bm{c}_g}=D=\ev{\bm{c}_1,\bm{c}_2}$. Then, there exist $\lambda_1,\lambda_2,\mu_1,\mu_2\in \Fn{q}$ such that $\bm{c}_f=\lambda_1\bm{c}_1+\lambda_2\bm{c}_2$ and $\bm{c}_g=\mu_1\bm{c}_1+\mu_2\bm{c}_2$.  By Proposition~\ref{out-c1}, we have $\lambda_2=\mu_2=0$. However,  this contradicts the fact that $\dim\ev{\bm{c}_f,\bm{c}_g}=2$.
\end{proof}

\section{Minimal free resolutions and $d_2$-test sets}\label{s:free-res}

Let $\Co\subset\F{q}{n}$ be a linear code and $M\subset\crc{\Co}$. In this section, we show how the condition that $M$ is a $d_2$-test set is equivalent to determining the second GHW of $\Co$  from the Betti numbers of  the minimal free resolution of the monomial ideal associated to $M$. Here, we adopt the notation introduced in Subsection~\ref{s:pre-stanley}.

Let $\Co(n,k)_q$ be a linear code and $M\subset\crc{\Co}$. Let $\ideal{M}$ be the ideal associated to $M$, defined in \eqref{ideal-M-g}. First, we discuss how the Betti numbers of the minimal free resolution of $R/\ideal{M}$  are related to the Hamming weight of codewords in $M$ and to the weight of two-dimensional vector subspaces generated by codewords from $M$.  More specifically, we have the following:

\begin{Proposition}\label{lemma-d1-d2-test}
Let $\Co(n,k)_q$ be a linear code, and let $M\subset \crc{\Co}$ be such that $|S_M|\geq 2$. Then
\begin{enumerate}
\item\label{3-a} $\pd(R/\ideal{M})\leq k$.
\item\label{3-b} $\min \{j \mid \beta_{1,j}(R/\ideal{M})\neq 0\}=\min\{\w{c}:\bm{c}\in M\}$.
\item\label{3-c} $\min \{j \mid \beta_{2,j}(R/\ideal{M})\neq 0\}=\min \{\wev{\bm{c},\bm{c}'}: \bm{c},\bm{c}'\in M,\ \dim \ev{ \bm{c},\bm{c}'}=2\}$.
\end{enumerate}

\begin{proof}
We consider the minimal graded free resolution of $R/\ideal{M}$:
$$0 \longrightarrow F_{\rho} \stackrel{\partial_{\rho}}\longrightarrow \cdots \stackrel{\partial_{3}}\longrightarrow F_2 \stackrel{\partial_{2}}\longrightarrow F_1 \stackrel{\partial_{1}}\longrightarrow F_0=R \longrightarrow R / \ideal{M} \longrightarrow 0.$$

Since $\ideal{M}\subset\ideal{\crc{\Co}}$, the statement in~\ref{3-a} follows from Theorem~\ref{prop:betti-verdure}.

It remains to prove statements in~\ref{3-b} and~\ref{3-c}.
Let $S_M=\{\sigma_1,\sigma_2,\ldots,\sigma_r\}$. Then, $\ideal{M}=\langle \bm{x}^{\sigma_1},\bm{x}^{\sigma_2},\ldots,\bm{x}^{\sigma_r} \rangle$. Since $M\subset \crc{\Co}$, the set $\{\bm{x}^{\sigma_1},\bm{x}^{\sigma_2},\ldots,\bm{x}^{\sigma_r}\}$ is the minimal monomial generating set of $\ideal{M}$. Also, since $|S_M|\geq 2$, we have $\rho\geq 2$. 

Now, for each $\sigma_i\in S_M$, there exists $\bm{c}_i\in M$ such that $\suppa{\bm{c}_i}=\sigma_i$. Hence, $F_1=\textstyle\bigoplus_{j=1}^r R(-\deg(\bm{x}^{\sigma_j}))=\bigoplus_{j=1}^r R(-\wa{\bm{c}_j})$. Then 
 $$\min \{j \mid \beta_{1,j}(R/\ideal{M})\neq 0\}=\min\{\wa{\bm{c}_j}:j\in[r]\}=\min\{\w{c}:\bm{c}\in M\}.$$

In the remainder of the proof, we focus on proving item~\ref{3-c}.  Let $a=\min \{\wev{\bm{c},\bm{c}'}: \bm{c},\bm{c}'\in M,\ \dim \ev{ \bm{c},\bm{c}'}=2\}$. We want to prove that $\min \{j \mid \beta_{2,j}(R/\ideal{M})\neq 0\}=a$. Since $M\subset \crc{\Co}$, for  $i\neq j$, we have $\dim \ev{ \bm{c}_i,\bm{c}_j}=2$. From this and Remark~\ref{Remark-independiente}, it follows that $\{\supp{c}\cup \suppa{\bm{c}'}:\bm{c},\bm{c}'\in M,\ \dim \ev{ \bm{c},\bm{c}'}=2\}=\{\suppa{\bm{c}_i}\cup \suppa{\bm{c}_j}:1\leq i<j\leq r\}$. Then, $a=\min \{\{\wev{\bm{c}_i,\bm{c}_j}: 1\leq i<j\leq r\}$. 

On the other hand, let $\{e_1, \ldots, e_r\}$ be the canonical basis of $F_1$.  Then, the differential $\partial_1: F_1 \to F_0$ is given by
$$\partial_1(e_i) = \bm{x}^{\sigma_i} \quad \text{for } i \in [r].$$

Let $m_i=\bm{x}^{\sigma_i}$ and $m_{ij}=\lcm(\bm{x}^{\sigma_i},\bm{x}^{\sigma_j})$. We have $\ker \partial_1=\left\langle S_{ij}:1\leq i<j\leq r \right\rangle $, where $S_{ij}=\frac{m_{ij}}{m_i}e_i-\frac{m_{ij}}{m_j}e_j$. Then, $
\deg(S_{ij})=\deg(m_{ij})=\wev{\bm{c}_i,\bm{c}_j}$. Hence 
\begin{align}
a=\min \{\{\wev{\bm{c}_i,\bm{c}_j}: 1\leq i<j\leq r\}=\min \{\deg(S_{ij}): 1\leq i<j\leq r\}.\label{igua-1}
\end{align}
Let  $A$ be a minimal generating set of $\ker\partial_1$. To prove $\min \{j \mid \beta_{2,j}(R/\ideal{M})\neq 0\}=a$, it suffices to show that there exists an element in $A$ whose degree is equal to $a$.  Let $S_{ij}$ be such that $\deg(S_{ij})=a$. If $S_{ij}\in A$, we have the required result. If this is not the case, 
we can separate the elements in $A$ in those that have the canonical vector $e_i$, the canonical vector $e_j$, and those that do not have any of them. Hence, we can write $A=\{S_{ia_1},\ldots,S_{ia_p},S_{b_1j},\ldots,S_{b_\ell j},S_{i_1j_1},\ldots,S_{i_mj_m}\}$. Since $S_{ij}\in \ker \partial_1= \langle A\rangle$, there exists $f_1,f_2,\ldots,f_p,g_1,g_2,\ldots,g_\ell,h_1,h_2,\ldots,h_m\in R$ such that

\begin{align*}
S_{ij}&=\sum_{u=1}^pf_{u}S_{ia_{u}}+\sum_{u=1}^\ell g_{u}S_{b_{u}j}+\sum_{u=1}^m h_{u}S_{i_{u}j_u}\\
&=\sum_{u=1}^pf_{u}\left(\frac{m_{ia_u}}{m_i}e_i-\frac{m_{ia_u}}{m_{a_u}}e_{a_u}\right)+\sum_{u=1}^\ell g_{u}\left(\frac{m_{b_uj}}{m_{b_u}}e_{b_u}-\frac{m_{b_uj}}{m_j}e_j\right)+\sum_{u=1}^m h_{u}S_{i_{u}j_u}\\
&=\sum_{u=1}^pf_{u}\frac{m_{ia_u}}{m_i}e_i-\sum_{u=1}^\ell g_{u}\frac{m_{b_uj}}{m_j}e_j-\sum_{u=1}^pf_{u}\frac{m_{ia_u}}{m_{a_u}}e_{a_u}+\sum_{u=1}^\ell g_{u}\frac{m_{b_uj}}{m_{b_u}}e_{b_u}+\sum_{u=1}^m h_{u}S_{i_{u}j_u}
\end{align*}
Since $S_{ij}=\frac{m_{ij}}{m_i}e_i-\frac{m_{ij}}{m_j}e_j$, we have $m_{ij}=\sum_{u=1}^pf_{u}m_{ia_u}$. Hence, $m_{ij}\in \langle m_{ia_1},\ldots,m_{ia_{p}}\rangle$. Then, there exists $m_{ia_u}$ such that $m_{ia_u}\mid m_{ij}$. Hence, $\deg(m_{ia_u})\leq \deg (m_{ij})=a$. Also, since $A\subset \{S_{ij}:1\leq i<j\leq r\}$, by \eqref{igua-1} we have $\deg(S_{ia_u})=\deg(m_{ia_u})\geq a$. Hence, $\deg(S_{ia_u})=a$, where $S_{ia_u}\in A$. 
\end{proof}
\end{Proposition}

As a consequence of Proposition~\ref{lemma-d1-d2-test}, we obtain our third main result:

\begin{Teorema}\label{ym-free-re}
Let $\Co(n,k)_q$ be a linear code, and let $M\subset \crc{\Co}$ be such that $|S_M|\geq 2$.  Then
\begin{enumerate}
\item $\min \{j \mid \beta_{1,j}(R/\ideal{M})\neq 0\}=d_1(\Co)$ if and only if $M$ contains a codeword of minimum Hamming weight.
\item $\min \{j \mid \beta_{2,j}(R/\ideal{M})\neq 0\}=d_2(\Co)$ if and only if $M$ is a $d_2$-test set.
\end{enumerate}
\end{Teorema}

From the previous result, Proposition~\ref{prop:d1-Mg} and Theorem~\ref{corollary-main}, we have the following:
\begin{Corollary}
Let $\Co\subset\F{q}{n}$ be a linear code and let $\G$ be the reduced Gr\"obner basis of $\Ic$. The following holds:
\begin{enumerate}
\item $d_1(\Co)=\min \{j \mid \beta_{1,j}(R/\ideal{M_\G})\neq 0\}$.
\item If $|\I \cap \J|\leq \tfrac{|\J|+1}{2}$, then $d_2(\Co)=\min \{j \mid \beta_{2,j}(R/\ideal{M_\G})\neq 0\}$.
\end{enumerate}
\end{Corollary}

In this study, we extend the findings of  \citet{garcia2022free} by introducing the notion of a $d_2$-test set for a linear code $\Co\subset\F{q}{n}$, which constitutes a more refined structure from which $d_2(\Co)$ can be determined. While in the binary case the set $M_\G$ is a $d_2$-test set, our results show that, if $\Co\subset\F{q}{n}$, this is not generally true when $q>2$. Consequently,  the results presented in Theorems \ref{main-result} and \ref{ym-free-re} highlight that the information contained in the minimal free resolution of $R/I_{M_\G}$ is, in general, insufficient to recover $d_2(\Co)$ as in the binary case.

\subsection*{Acknowledgments}  
We would like to thank Daniel Duarte (Centro de Ciencias Matem\'aticas, Universidad Nacional Aut\'onoma de M\'exico, M\'exico) for his valuable guidance, insightful comments and suggestions throughout this work.

\bibliography{demo}
\end{document}